\documentclass[12pt,letterpaper]{article}

\usepackage{amssymb,amsfonts,amscd,amsthm}
\usepackage[all,arc]{xy}
\usepackage{bbm}
\usepackage{mathrsfs}
\usepackage{tikz}
\usepackage[left=0.9in,top=0.9in,right=0.9in,bottom=0.9in]{geometry}
\usepackage{mathtools}
\usepackage{hyperref}
\usepackage{cleveref}
\usepackage{color}
\usepackage{graphicx}
\usepackage{enumerate}
\graphicspath{ {TexImages/} }
\usepackage{enumitem}
\usepackage{subcaption}

\newtheorem{theorem}{Theorem}[section]
\newtheorem{corollary}[theorem]{Corollary}
\newtheorem{proposition}[theorem]{Proposition}
\newtheorem{lemma}[theorem]{Lemma}
\newtheorem{claim}[theorem]{Claim}

\newtheorem{question}[theorem]{Question}

\newtheorem{conjecture}[theorem]{Conjecture}
\newtheorem{observation}[theorem]{Observation}

\theoremstyle{definition}
\newtheorem{definition}{Definition}

\hypersetup{
	colorlinks,
	citecolor=blue,
	filecolor=blue,
	linkcolor=blue,
	urlcolor=blue,
	linktocpage
}

\title{}
\author{}
\date{}

\newcommand{\pow}{q} 
\newcommand{\Pow}{{q_\star}} 
\newcommand{\re}{r} 
\newcommand{\ex}{\mathrm{ex}}
\renewcommand{\c}[1]{\mathcal{#1}}
\newcommand{\Om}{\Omega}

\newcommand{\Del}{\Delta}
\newcommand{\al}{\alpha}
\newcommand{\sm}{\setminus}
\newcommand{\sub}{\subseteq}
\newcommand{\E}{\mathbb{E}}
\newcommand{\mon}{\mathrm{mon}}
\newcommand{\Mon}{\mathrm{Mon}}

\title{Rational Exponents for General Graphs}

\author{Sean English\thanks{University of North Carolina Wilmington, \texttt{EnglishS@uncw.edu}.} \and Sam Spiro\thanks{Georgia State University, \texttt{sspiro@gsu.edu}.}}

\begin{document}
	
	\maketitle

	\begin{abstract}
		A rational number $r$ is a \textbf{realizable exponent} for a graph $H$ if there exists a finite family of graphs $\mathcal{F}$ such that $\mathrm{ex}(n,H,\mathcal{F})=\Theta(n^r)$, where $\mathrm{ex}(n,H,\mathcal{F})$ denotes the maximum number of copies of $H$ that an $n$-vertex $\mathcal{F}$-free graph can have. Results for realizable exponents are currently known only when $H$ is either a star or a clique, with the full resolution of the $H=K_2$ case being a major breakthrough of Bukh and Conlon.  
		
		In this paper, we establish the first set of results for realizable exponents which hold for arbitrary graphs $H$ by showing that for any graph $H$ with maximum degree $\Delta \ge 1$, every rational in the interval $\left[v(H)-\frac{e(H)}{2\Delta^2},\ v(H)\right]$ is realizable for $H$. We also prove a ``stability'' result for generalized Tur\'an numbers of trees which implies that if $T\ne K_2$ is a tree with $\ell$ leaves, then $T$ has no realizable exponents in $[0,\ell]\setminus \mathbb{Z}$.  Our proof of this latter result uses a new variant of the classical Helly theorem for trees, which may be of independent interest.
	\end{abstract}

	\section{Introduction}
	Given a family of graphs $\c{F}$, the \textbf{Tur\'an number} $\ex(n,\c{F})$ is defined to be the maximum number of an edges that an $n$-vertex $\c{F}$-free graph can have.  When $\c{F}$ consists of a single graph $F$ we denote its Tur\'an number simply by $\ex(n,F)$.  As stated $\ex(n,\c{F})$ is not well-defined if $\c{F}$ contains an empty graph on at most $n$ vertices; in such cases we adopt the convention that $\ex(n,\c{F})=0$.
	
	The behavior of $\ex(n,\c{F})$ is well-understood whenever $\c{F}$ contains only non-bipartite graphs due to the famed Erd\H{o}s-Stone-Simonovits Theorem~\cite{ES1946}, but the behavior of $\ex(n,\c{F})$ when $\c{F}$ contains a bipartite graph is far from clear despite decades of active research into the problem, and we refer the reader to the survey~\cite{FS2013} for a more thorough summary of what is known in the area.
	
	In an attempt to better understand the possible behaviors of $\ex(n,\c{F})$, Erd\H{o}s and Simonovits made a conjecture, now commonly known as the rational exponents conjecture, which states that for every rational $\re\in [1,2]$ there exists a graph $F$ such that $\ex(n,F)=\Theta(n^\re)$.  A positive answer to a slight weakening of this conjecture was given  in breakthrough work of Bukh and Conlon~\cite{BC2018} where the following was proven.
	
	\begin{theorem}[\cite{BC2018}]\label{theorem bukh conlon}
		For every rational $\re\in [1,2]$ there exists a finite family of graphs $\c{F}$ such that $\ex(n,\c{F})=\Theta(n^r)$.
	\end{theorem}
	
	\Cref{theorem bukh conlon} is a highly influential result, with it serving as the catalyst for a large body of literature dedicated towards proving special cases of the rational exponents conjecture~\cite{E1981}.  It is also worth mentioning that \Cref{theorem bukh conlon} was proven using the powerful random polynomial method introduced  by Bukh~\cite{B2015}, with this method having since been further developed to solve a large number of problems in combinatorics, see e.g. \cite{BKN2022, CPZ2021, CT2021, F2019, KLZ2024, ST2024}.
	
	We aim to study a generalization of the rational exponents conjecture to the generalized Tur\'an setting, a problem which was initiated in recent work of English, Halfpap, and Krueger~\cite{EHK2025}.  To this end, given a graph $H$ and a family of graphs $\c{F}$, we define the \textbf{generalized Tur\'an number} $\ex(n,H,\c{F})$ to be the maximum number of copies of $H$ that can appear in an $n$-vertex $\c{F}$-free graph.  The systematic study of generalized Tur\'an numbers was initiated by Alon and Shikhelman~\cite{AS2016}, and since then there has been a large number of results dedicated to the topic, see e.g. \cite{GGMV2020, GP2022, HP2021, MYZ2018, ZGHLSX2023}.  For more on the history of generalized Tur\'an numbers we refer the interested reader to the excellent recent survey by Gerbner and Palmer~\cite{GP2025}.  
	
	Motivated by \Cref{theorem bukh conlon}, we consider the following.
	
	\begin{definition}
		We say that a rational number $\re\in [0,v(H)]$ is a \textbf{realizable exponent} (or simply \textbf{realizable}) for a graph $H$ if there exists a finite family of graphs $\c{F}$ satisfying $\ex(n,H,\c{F})=\Theta(n^\re)$.  We say that $\re$ is \textbf{non-realizable} for $H$ if it is not a realizable exponent.
	\end{definition}
	
	As we will see later on, every graph $H$ without isolated vertices has $0$ as a realizable exponent and no rational in $(0,1)$ being realizable, so the central problem will be in studying the range $\re\in [1,v(H)]$.
	
	Note that $\ex(n,K_2,\c{F})=\ex(n,\c{F})$, so \Cref{theorem bukh conlon} is equivalent to saying that every rational in $[1,2]$ is realizable for $H=K_2$, showing that every rational which could be realizable for $K_2$ is realizable.  It is natural to ask to what extent this holds for other $H$, and in particular what happens for families of $H$ which contain $K_2$ as a member.  For example, one might ask what happens when $H$ is a clique, with this case being solved in full by  English, Halfpap, and Krueger~\cite{EHK2025}.
	
	\begin{theorem}[\cite{EHK2025}]\label{theorem:cliques}
		For all $t\ge 2$, every rational in $[1,t]$ is realizable for $H=K_t$.
	\end{theorem}
	
	Thus as we saw for $K_2$, every rational which could be realizable for cliques is  realizable. In view of this, one might be led to guess that $[1,v(H)]$ should be realizable for every graphs.  Perhaps surprisingly, this phenomenon does not in fact hold, and it in particular fails for all stars larger than $K_2$.
	
	\begin{theorem}\label{theorem:stars}
		For all $t\ge 2$, every rational in $[t-1,t]$ is realizable for $H=K_{1,t-1}$, but no rational in $(1,t-1)$ is realizable.
	\end{theorem}

	\section{Main Results}
	In this paper, we study realizable exponents for arbitrary graphs, proving several general results about both realizable and non-realizable exponents, with us significantly improving upon these results whenever $H$ is a forest.  Moreover, our tools used to give these improvements for forests can be used to give a number of other results on generalized Tur\'an numbers of trees which are of independent interest.
	
	\subsection{Realizable Exponents} 
	
	Our first main result shows that for \textit{any} graph $H$, every sufficiently large rational is realizable for $H$.
	\begin{theorem}\label{theorem:maxDegreeGeneral}
		If $H$ is a graph with maximum degree $\Delta\ge 1$, then every rational number in the interval 
		\[\left[v(H)-\frac{e(H)}{2\Delta^2},\ v(H)\right]\]
		is realizable for $H$.
	\end{theorem}
	
	A more careful analysis of our methods can be used to establish an even stronger result for trees.
	\begin{theorem}\label{theorem:maxDegreeTrees}
		If $T$ is a tree with maximum degree $\Delta\ge 1$, then every rational number in the interval 
		\[\left[v(T)-\frac{e(T)}{\Delta},\ v(T)\right]\]
		is realizable for $T$.
	\end{theorem}
	The length of the interval in \Cref{theorem:maxDegreeTrees} is best possible for trees in general, as \Cref{theorem:stars} shows that one can not improve this for stars.


	\subsection{Non-Realizable Exponents}  
	
	Our main contribution to the study of non-realizable exponents is the establishment of a technical result, \Cref{proposition:keyObservation}, which, at present, is the only tool we know of that can prove the non-existence of realizable exponents.  We postpone stating this result due to its technical nature and instead focus on some of its consequences, starting with the following non-realizable exponents of arbitrary graphs.

	\begin{proposition}\label{proposition no 01 exponents}
		No rational in $(0,1)$ is realizable for any graph $H$.  Moreover, $0$ is realizable for every graph $H$ without isolated vertices, and $1$ is realizable for every connected graph $H$.
	\end{proposition}
	The condition that $H$ is connected for $1$ to be realizable is necessary in general, with the simplest example being any $H$ an independent set of size at least 2.
	
	Observe by \Cref{theorem:cliques} that the only rationals which are non-realizable for cliques lie in $(0,1)$, so \Cref{proposition no 01 exponents} is best possible for arbitrary graphs.  On the other hand, the set of realizable exponents for stars detailed in \Cref{theorem:stars} suggests that we might be able to say more about non-realizable exponents of trees.  And indeed, in this setting we have the following.
	\begin{theorem}\label{theorem:leafNonrealizable}
		If $T\ne K_2$ is a tree with $\ell\ge 2$ leaves, then no rational in $[0,\ell]\setminus \mathbb{Z}$ is realizable for $T$.
	\end{theorem}
	\Cref{theorem:stars} shows that the interval $[0,\ell]$ is best possible for stars.  We believe it is possible to push our methods further to completely characterize the set of realizable exponents for every tree in the range $[0,\ell]$, but due to its additional complexity, we defer this problem to future work.  
	
	\subsection{Generalized \texorpdfstring{Tur\'an}{Turan} Numbers of Trees}
	
	The methods we use here give several general results regarding the generalized Tur\'an number for trees.  The only prior results about counting copies of arbitrary trees $T$ that we are aware of are results of Gerbner~\cite{G2023} who determined the order of magnitude of $\ex(n,T,K_{2,t})$ for all trees $T$ and a result of Cambie, de Joannis de Verclos, and Kang~\cite{CJK23} who observed the exact value of $\ex(n,T,K_{1,t})$.  The only specific choice of $T$ for which $\ex(n,T,\c{F})$ seems to be well understood in general is when $T$ is a star, in which case general results of F\" uredi and K\" undgen~\cite{FK2006} gives an essentially optimal bound on $\ex(n,K_{1,t},\c{F})$ based on the classical Tur\'an number $\ex(n,\c{F})$.   In a similar spirit, the following result gives an effective relationship between $\ex(n,T,\c{F})$ and $\ex(n,\c{F})$ for arbitrary trees.
	
	
	\begin{theorem}\label{theorem weak tree reduction}
		If $T\ne K_1$ is a tree and if $\c{F}$ is a family of graphs, then for every integer $k\ge 1$ either
		\[
		\ex(n,T,\c{F})=\Om_{T,k}(n^k),
		\]
		or
		\[
		\ex(n,T,\c{F})=O_{T,k,\c{F}}(\ex(n,\c{F})^{k-1}).
		\]
	\end{theorem}
	In fact, we prove a slightly stronger technical version of this statement in \Cref{theorem:treeReduction} which characterizes when the lower bound or upper bound applies to a given family $\c{F}$.  We note that the statement of \Cref{theorem weak tree reduction} does not typically hold if one replaces the tree $T$ with some other type of graph $H$. In particular, such an extension fails for $H=C_4$; see the discussion at the end of Section~\ref{section concluding remarks} for more.
	
	We will show that \Cref{theorem weak tree reduction} can be used to quickly imply the following, which in essence states that the generalized Tur\'an problem for trees vs forests always has a very simple structure.
	\begin{corollary}\label{corollary trees vs forests}
		If $T\ne K_1$ is a tree and if $\c{F}$ is a family of graphs which contains a forest, then either $\ex(n,T,\c{F})=\Theta(n^k)$ for some integer $k\le \al(T)$ or $\ex(n,T,\c{F})=0$ for all sufficiently large $n$.
	\end{corollary}

	This in turn implies the following ``stability'' result for generalized Tur\'an numbers of trees, which informally says that either $\ex(n,H,\c{F})$ is large or $\ex(n,H,\c{F})$ takes on a very particular form.  As far as we are aware this is the first such ``stability'' result of its kind for generalized Tur\'an numbers.
	
	\begin{corollary}\label{theorem:leafGap}
		If $T\ne K_2$ is a tree with $\ell\ge 2$ leaves and if $\c{F}$ is a family of graphs such that $\ex(n,T,\c{F})=O(n^\ell)$, then either $\ex(n,T,\c{F})=\Theta(n^k)$ for some integer $k$ or $\ex(n,T,F)=0$ for all sufficiently large $n$.
	\end{corollary}
	
	Observe that \Cref{theorem:leafNonrealizable} immediately follows from \Cref{theorem:leafGap}.
	
	\subsection{Organization}

	We establish some general preliminary results and notation in \Cref{section preliminaries}.  We prove our realizable results in \Cref{section realizable} and both our non-realizable results and generalized Tur\'an results for trees in \Cref{section non-realizable}.  We emphasize that these two sections are independent of each other and can be read in either order.  We conclude with some comments and open problems in \Cref{section concluding remarks}.

	\section{Preliminaries}\label{section preliminaries}
	In this section we gather our necessary tools and definitions.  We will begin by shifting our perspective away from trying to count copies of $H$ and instead counting the following closely related concept.
	
	\begin{definition}
		Given graphs $H,G$, we say that a map  $\phi:V(H)\to V(G)$ is a \textbf{monomorphism} of $H$ into $G$ if $\phi$ is both injective and a homomorphism.  We let $\Mon(H,G)$ denote the set of monomorphisms of $H$ into $G$ and we let $\mon(H,G):=|\Mon(H,G)|$.
	\end{definition}
	The main motivation for this definition is the following.
	\begin{observation}\label{observation monomorphisms same as copies}
		For every pair of graphs $H,G$, the number of copies of $H$ in $G$ equals
		\[\frac{\mon(H,G)}{\mathrm{aut}(H)},\]
		where $\mathrm{aut}(H)$ denotes the size of the automorphism group of $H$, i.e.\ the number of isomorphisms from $H$ to itself.
	\end{observation}
	Indeed, this follows from the fact that a copy of $H$ in $G$ is just a subgraph of $G$ which is isomorphic to $H$, and the exactly $\mathrm{aut}(H)$ choices of isomorphisms for each copy gives rise to exactly $\mathrm{aut}(H)$ monomorphisms corresponding uniquely to this copy.  Because of this observation, we will often bound the number of copies of a graph $H$ by instead bounding $\mon(H,G)$.

	We next introduce the notion of rooted graphs which will be crucial to our study of realizable exponents.  While these graphs will not show up explicitly in our work on non-realizable exponents, they will turn out to be quite related to the objects needed for our main non-realizable result \Cref{proposition:keyObservation}.
	
	\begin{definition}
		Given a graph $F$ and set of vertices $R\subsetneq V(F)$, we say that the pair $(F,R)$ is a \textbf{rooted graph} and that $R$ is its set of \textbf{roots}.  We define the \textbf{rooted density} of a rooted graph to be
		\[
		\max_{S\sub V(F)\sm R,\ S\ne \emptyset}\frac{\mathrm{inc}(S)}{|S|},
		\]
		where $\mathrm{inc}(S)$ is defined to be the number of edges of $F$ incident to a vertex of $F$.  We say that $(F,R)$ is \textbf{balanced} if it has rooted density equal to $\frac{e(F)}{v(F)-|R|}$.
		
		Given a rooted graph $(F,R)$ and a positive integer $\pow$, we define the \textbf{rooted power} $(F,R)^\pow$ to be the family of graphs $F'$ which can be written as the union of $\pow$ distinct copies of $F$ which all agree on the set $R$. 
	\end{definition}
	
	The fundamental fact we need about these definitions is that there exists graphs which contain many copies of $H$ while avoiding rooted powers.  This was initially proven for $H=K_2$ in a highly influential result of Bukh and Conlon~\cite{BC2018} with this later being extended in \cite{S24} to general $H$.

	\begin{theorem}[Proposition 2.1 \cite{S24}]\label{theorem random polynomial lower bounds}
		If $H$ is a graph and if $(F_1,R_1),\ldots,(F_t,R_t)$ are rooted graphs which all have rooted density at least $d$ for some rational $d$, then there exists some integer $\pow_0$ such that for all integers $\pow\ge \pow_0$, we have
		\[\ex(n,H,\bigcup_i (F_i,R_i)^\pow)=\Omega(n^{v(H)-e(H)/d}).\]
	\end{theorem}
	
	\Cref{theorem random polynomial lower bounds} is most effective when at least one of the rooted graphs is balanced.  In order to build such balanced rooted graphs, we will rely on the following fact proven by Bukh and Conlon that balanced rooted trees exist for all densities.
	
	\begin{lemma}[Lemma 1.3 \cite{BC2018}]\label{lem:BukhConlon}
		For all positive integers $a,b$ with $b\ge a-1$, there exists a balanced rooted tree $(T,R_T)$ of rooted density $b/a$ which has $a$ unrooted vertices $u_1,\ldots,u_a$ forming a path in $T$ with $\deg_T(u_i)\ge b/a$ for all $i$.
	\end{lemma}
	We note that the condition $\deg_T(u_i)\ge b/a$ is not immediate from the definition of $(T,R_T)$ given in \cite[Definition 1.5]{BC2018}, but this condition follows from the fact that these rooted trees are balanced with rooted density $b/a$.

	We will often be counting copies of $H$ in some larger host graph $G$. To distinguish between vertices in $H$ and copies of these vertices contained in a subgraph of $G$ isomorphic to $H$, we will often write $\hat{x},\hat{y}$, etc. for vertices that are in $H$ with us omitting the ``hat'' notation for vertices in $G$ (usually with $x\in V(G)$ playing the role of $\hat{x}\in V(H)$).

	\section{Realizable Exponents}\label{section realizable}
	
	We  make use of the following technical definition to show the existence of realizable exponents.
	\begin{definition}
		We say that a graph $H$ is \textbf{$d$-admissible} for some rational $d$ if there exists a finite set of rooted graphs $\{(F_1,R_1),\ldots,(F_t,R_t)\}$ which each have rooted density at least $d$ such that for all $\pow\in\mathbb{N}$,  we have 
		\[
		\mathrm{ex}(n,H,\bigcup_i(F_i,R_i)^\pow)=O_\pow(n^{v(H)-e(H)/d}).
		\]
		In this case we say $H$ is $d$-admissible with respect to the set $\{(F_1,R_1),\ldots,(F_t,R_t)\}$.
	\end{definition}
	
	The key observation with this definition is the following.
	\begin{observation}\label{observation admissible implies realizable}
		If $H$ is $d$-admissible, then the rational $v(H)-e(H)/d$ is realizable for $H$.
	\end{observation}
	Indeed, this follows by taking $\c{F}_\pow$ to be the family of $\pow$-th powers of the rooted graphs $(F_i,R_i)$ for which $H$ is $d$-admissible with respect to, as such a family satisfies $\ex(n,H,\c{F}_\pow)=O_\pow(n^{v(H)-e(H)/d})$ by definition, and for $\pow$ large enough it also satisfies $\ex(n,H,\c{F}_\pow)=\Omega(n^{v(H)-e(H)/d})$ by \Cref{theorem random polynomial lower bounds}.
	
	With \Cref{observation admissible implies realizable} in mind, proving realizability reduces to proving admissibility.  The main advantage of this perspective is that the notion of admissibility turns out to be quite malleable for inductively building larger admissible graphs from smaller ones.  For example, we have the following basic fact.
	
	\begin{observation}\label{observation disjoint union of dadmissible}
		If $H_1$ and $H_2$ are $d$-admissible, then their disjoint union $H_1\sqcup H_2$ is also $d$-admissible.
	\end{observation}
	Indeed, this follows by taking the union of the two sets of rooted graphs for which $H_1,H_2$ are $d$-admissible with respect to. 

	With \Cref{observation disjoint union of dadmissible}, it suffices to determine admissibility of connected graphs.  In the original work of Bukh and Conlon for $H=K_2$, this was done by establishing \Cref{lem:BukhConlon} showing the existence of certain balanced rooted trees.  In a similar spirit, we will construct balanced rooted graphs which are ``tree-like'' with respect to $H$.
	
	\begin{definition}\label{definition H-tree}
		Given a graph $H$, we say that a graph $F$ is an \textbf{$H$-tree} if there exist distinct copies $H_1,\ldots,H_t$ of $H$ in $F$ together with isomorphisms $h_i:V(H)\to V(H_i)$ for all $i$ such that the following conditions hold:
		\begin{enumerate}[label=(\Alph*)]
			\item We have $F=\bigcup_{i=1}^t H_i$, and
			\item For all $i\ge 2$, there exists an integer $1\le j_i<i$ such that 
			\[V(H_i)\cap \bigcup_{j=1}^{i-1} V(H_j)=V(H_i)\cap V(H_{j_i}),\]
			and such that every $v\in V(H_i)\cap V(H_{j_i})$ satisfies $h^{-1}_i(v)=h^{-1}_{j_i}(v)$.
		\end{enumerate}
		The set $\{H_1.\ldots,H_t\}$ will be called a set of \textbf{defining copies} of $F$.  The sequence $(h_1,j_2,h_2,j_3,\dots,h_t)$ will be called a \textbf{witness} for $F$.  We say that a rooted graph $(F,R)$ is a \textbf{rooted $H$-tree} if $F$ is an $H$-tree.
	\end{definition}
	We emphasize that the definition of an $H$-tree theoretically allows for $F$ to be disconnected, but in practice we only ever needs to use connected $F$, hence our naming this an $H$-tree rather than an $H$-forest.
	
	We can now state our main result for $d$-admissible graphs for $d$ large.  For this statement, we say that $H'\sub H$ is a \textbf{proper induced subgraph} of $H$ if $H'$ is an induced subgraph with $H'\ne H$.

	\begin{theorem}\label{theorem d large tree implies dadmissible}
		Let $H$ be a connected graph and $d\geq \Delta(H)$ rational. If there exists a balanced rooted $H$-tree $(F,R)$ of rooted density $d$ such that every proper induced subgraph of $H$ and every component of $F[R]$ is $d$-admissible, then $H$ is $d$-admissible.
	\end{theorem}
	We emphasize that \Cref{theorem d large tree implies dadmissible} reduces the a priori complicated problem of showing that a given (large) rational $v(H)-e(H)/d$ is realizable to the purely combinatorial problem of constructing some fixed rooted graph $(F,R)$ with a few special properties.  \Cref{theorem d large tree implies dadmissible} in fact holds even without the assumption that $H$ is connected, but there is no need to consider disconnected graphs due to \Cref{observation disjoint union of dadmissible}.
	
	We prove \Cref{theorem d large tree implies dadmissible} in Subsection~\ref{subsection proof admissible}, after which we use \Cref{theorem d large tree implies dadmissible} in Subsection~\ref{subsection building H-trees} to prove our main results for realizable exponents.  These two subsections are self-contained and can be read independently of each other.  In what follows we will make use of the following notation.
	
	\begin{definition}
		Given a graph $H$, we let $2^H$ denote the set of induced subgraphs of $H$. 
	\end{definition}
	
	In particular, this means $2^H\sm \{H\}$ is shorthand for the set of proper induced subgraphs of $H$. 
	
	We close with a simple observation that will be the base case of most of our inductive arguments. 
	
	\begin{observation}\label{observation K1 is easy}
		The graph $K_1$ is $d$-admissible for all rationals $d>0$.
	\end{observation}
	
	Indeed, any non-empty rooted graph $(F,R)$ of rooted density at least $d$ satisfies $\ex(n,K_1,(F,R)^\pow)=n=O(n^{v(K_1)-e(K_1)/d})$.

	\subsection{Proof of \texorpdfstring{\Cref{theorem d large tree implies dadmissible}}{Theorem 4.3}}\label{subsection proof admissible}
	
	Let us start by discussing the high level proof ideas.  Ultimately, we will show that the graphs $H$ as in \Cref{theorem d large tree implies dadmissible} are $d$-admissible with respect to a family of rooted graphs of the form $(F,R)\cup \c{F}$ where $(F,R)$ is a suitably chosen rooted $H$-tree and $\c{F}$ is a set of rooted graphs for which every proper induced subgraph of $H$ and every component of $F[R]$ is $d$-admissible with respect to $\c{F}$.  To do this, we will show that if $G$ is a graph which avoids all $\pow$-th powers of graphs in $\c{F}$ and has many copies of $H$, then $G$ contains a $q$-th power of $(F,R)$, implying that graphs avoiding all $q$-th powers of graphs in $(F,R)\cup \c{F}$ can not contain many copies of $H$, proving that $H$ is $d$-admissible with respect to this family.
	
	We use a pigeonhole argument to show that graphs $G$ as above contain an element of $(F,R)^\pow$.  Specifically, we show that $G$ both contains few copies of $F[R]$ while also containing many copies of $F$, which means some $\pow$ of these copies of $F$ must agree on $R$.  Showing that the number of copies of $F[R]$ is small follows almost immediately from $G$ avoiding all $\pow$-th powers of graphs in $\c{F}$ and the fact that each component of $F[R]$ is $d$-admissible with respect to $\c{F}$.  The bulk of the work then lies in showing that $G$ contains many copies of the $H$-tree $F$.  In order to do this, we use the fact that $G$ has many copies of $H$ to show that there exists a large collection  of ``nice'' monomorphisms $\Phi$ of $H$ into $G$ which we refer to as a strong builder.  This strong builder is defined in such a way that we can glue together monomorphisms in $\Phi$ in many ways to form a copy of our $H$-tree $F$, with this process being analogous to the standard way one shows that $n$-vertex graphs with minimum degree $\delta$ contain at least $\Omega(\delta^en)$ copies of any fixed tree with $e$ edges. The framework we use to do this is a generalization of the framework developed in \cite{EHK2025} for cliques.
	
	With the above all in mind, our first goal will be to show that ``strong builders'' as described above exist in the graphs $G$ that we care about, after which we show that such strong builders do indeed correspond to the existence of many copies of $F$.  We begin with some definitions.

	\begin{definition}
		Let $H$ and $G$ be graphs and let $\Phi$ be a collection of monomorphisms from $H$ to $G$. Given an induced subgraph $H'\subseteq H$, we write $\Phi_{H'}$ to denote the collection $\{\phi|_{V(H')}: \phi\in \Phi\}$, i.e.\ $\Phi_{H'}$ is the set of restrictions of maps in $\Phi$ to $H'$. Given $\psi\in \Phi_{H'}$, the \textbf{degree} of $\psi$, denoted $\deg_\Phi(\psi)$, is the number of $\phi\in\Phi$ such that $\phi|_{V(H')}=\psi$.  We will denote the degree $\deg_\Phi(\phi)$ simply by $\deg(\phi)$ whenever $\Phi$ is clear from context. Given a set $S\subseteq V(G)$, let $\Phi-S$ denote the subset of $\Phi$ containing the monomorphisms $\phi$ with $\phi(V(H))\subseteq V(G)\setminus S$.
	\end{definition}
	
	\begin{definition}
		Let $H$ be a graph and $\beta:2^H\setminus \{H\}\to \mathbb{R}$ a real-valued function on the proper induced subgraphs of $H$.  We say that a collection of monomorphisms $\Phi$ from $H$ to a graph $G$ is a \textbf{$(H,\beta)$-weak builder} if for every $H'\in 2^H\setminus\{H\}$ and every $\psi\in \Phi_{H'}$, we have $\deg(\psi)\geq \beta(H')$. Given $s\in\mathbb{N}$, we say that $\Phi$ is a \textbf{$(H,\beta,s)$-strong builder} if both $\Phi-S$ is a $(H,\beta)$-weak builder and $\Phi_{H'}-S\subseteq (\Phi-S)_{H'}$ for all $S\subseteq V(G)$ with $|S|\leq s$ and $H'\in 2^H$ (i.e. if every $\psi\in \Phi_{H'}-S$, has at least one $\phi\in \Phi-S$ with $\phi|_{V(H)}=\psi$).
	\end{definition}
	
	Intuitively, $\Phi$ is a strong builder if there exist many ways of ``extending'' the restriction of any map $\phi\in \Phi$ even if there is a small set of vertices $S$ that these extensions are required to avoid.  We begin by showing that if there exists a large and ``well-controlled'' collection of monomorphisms $\Phi$, then this can be trimmed into a slightly smaller collection with a $(H,\beta)$-weak builder for some appropriate choice of $\beta$.  

	\begin{lemma}\label{lemma weak builder}
		Let $H$ and $G$ be graphs with $v(G)=n$, let $p$ be a real number, and let $\Phi$ be a collection of monomorphisms from $H$ to $G$. If
		\begin{itemize}
			\item $|\Phi|\geq 2^{v(H)+1}n^{v(H)}p^{e(H)}$, and
			\item $|\Phi_{H'}|\leq n^{v(H')}p^{e(H')}$ for all $H'\in 2^H\setminus\{H\}$,
		\end{itemize}
		then there exists a subcollection $\Phi'\subseteq \Phi$ such that
		\begin{itemize}
			\item $|\Phi'|\geq n^{v(H)}p^{e(H)}$, and
			\item For each $H'\in 2^H\setminus\{H\}$ and $\psi\in\Phi'_{H'}$, we have $\deg_{\Phi'}(\psi)\geq 2n^{v(H)-v(H')}p^{e(H)-e(H')}$.
		\end{itemize}
	\end{lemma}
	\begin{proof}
		Iteratively, as long as there exists some $H'\in 2^H\setminus\{H\}$ and $\psi\in \Phi_{H'}$ satisfying $\deg_{\Phi}(\psi)<2p^{e(H)-e(H')}n^{v(H)-v(H')}$, we remove from $\Phi$ any monomorphisms $\phi$ with $\phi|_{V(H')}=\psi$. Let $\Phi'$ denote the resulting set of monomorphisms after this removal process is done.  It remains only to show that $|\Phi'|\ge p^{e(H)} n^{v(H)}$.
		
		Observe that for a fixed $H'\in 2^H\setminus\{H\}$, if some $\psi\in \Phi_{H'}$ had too low of degree, then we removed at most $2p^{e(H)-e(H')}n^{v(H)-v(H')}$ monomorphisms which restrict to $\psi$ from $\Phi$.  As such, the total number of monomorphisms that could have been removed going from $\Phi$ to $\Phi'$ is at most 
		\[
		\sum_{H'\in 2^H\sm \{H\}} 2p^{e(H)-e(H')}n^{v(H)-v(H')}\cdot |\Phi_{H'}|.
		\]
		By our hypothesis we have $|\Phi_{H'}|\le p^{e(H')} n^{v(H')}$ for all $H'$. Combining this with the expression above gives that the number of $\psi$ removed from $\Phi$ is at most $(2^{v(H)}-1)2p^{e(H)}n^{v(H)}$, so we conclude that $\Phi'$ has the desired size.
	\end{proof}
	
	We next show that weak builders with $\beta$ sufficiently large are also strong builders with slightly weaker parameters.  Here and in what follows, for a real-valued function $\beta$ and real number $c$, we abuse notation slightly by writing $c\beta$ to denote the function which maps each element $x$ of the domain of $\beta$ to $c\cdot \beta(x)$.
	
	\begin{lemma}\label{lemma strong builder}
		If $\Phi$ is a $(H,\beta)$-weak builder with $\beta(H-\hat{w})\geq \frac{v(H)}{1-2^{-1/s}}+s$ for all $\hat{w}\in V(H)$, then $\Phi$ is also a $(H,\frac{1}{2}\beta,s)$-strong builder.
	\end{lemma}
	
	\begin{proof}
		We first show that for any set $S\sub V(G)$ of cardinality at most $s$ and any $H'\in 2^H$ that $\Phi_{H'}-S\subseteq (\Phi-S)_{H'}$.   If $H'=H$, then $\Phi_{H'}-S=\Phi-S=(\Phi-S)_{H'}$, so we may assume $H'\neq H$ from now on.  Let $\psi\in \Phi_{H'}-S$ and let $\phi\in \Phi$ be a map with $\phi|_{V(H')}=\psi$ such that $|\phi(V(H))\cap S|$ is as small as possible.  If $\phi\in \Phi-S$ (i.e.\ if $|\phi(V(H))\cap S|=0$) then we are done, so we may assume there exists some $\hat{w}\in V(H)$ with $\phi(\hat{w})=a$ for some $a\in S$.  Note that we must have $\hat{w}\notin V(H')$ since $\phi$ extends $\psi\in \Phi_{H'}-S$.  Let $H'':=H-\hat{w}$ and $\xi:=\phi|_{V(H'')}$, noting that $\xi|_{V(H')}=\psi$ since $\hat{w}\notin V(H')$.  Because $\beta(H-\hat{w})>s$, there exist at least $s+1$ extensions of $\xi$ in $\Phi$, and in particular there must exist some $\phi'\in \Phi$ with $\phi'(\hat{w})\notin S$.  In this case $\phi'$ (being an extension of $\xi$) is an extension of $\psi$ with $|\phi'(V(H))\cap S|=|\phi(V(H))\cap S|-1$, a contradiction to how we chose $\phi$.  We conclude that $\phi\in \Phi-S$, proving  that $\Phi_{H'}-S\subseteq (\Phi-S)_{H'}$.
		
		To prove that $\Phi$ is a $(H,\frac{1}{2}\beta,s)$-strong builder it only remains to show that $\Phi-S$ is a $(H,\frac{1}{2}\beta)$-weak builder for all sets $S$ of size at most $s$.  To prove this, consider any $a\in V(G)$ and let $\eta:=1-2^{-1/s}$, so that $(1-\eta)^s=\frac{1}{2}$. We will prove that $\Phi-a$ is a $(H,(1-\eta)\beta)$-weak builder, from which $\Phi-S$ being a $(H,\frac{1}{2}\beta)$-weak builder  will follow by iterating this a total of $s$ times.  
		
		Let $H'\in 2^H\setminus\{H\}$. Consider some $\psi\in (\Phi-a)_{H'}$, which we recall in our notation means that $\psi$ is the restriction of some $\phi\in \Phi$ which does not contain $a$ in its image to $V(H')$.  Our goal is to show that $\deg_{\Phi-a}(\psi)\geq (1-\eta) \beta(H')$, which will give the desired result mentioned above.
		
		If there exist at most $\eta \beta(H')$ full extensions of $\psi$ in $\Phi$ which include $a$ in their image, then we are done since $\deg_\Phi(\psi)\ge \beta(H')$ by hypothesis. If not, then there exists some collection $\Lambda\subseteq \Phi$ of at least $\eta \beta(H')$ monomorphisms from $H$ to $G$ such that every $\phi\in \Lambda$ is an extension of $\psi$ that contains $a$ in its image, i.e.\ such that $\phi|_{V(H')}=\psi$ and $a\in \phi(V(H))$.
		
		For each $\hat{w}\in V(H)\setminus V(H')$, let $\Lambda^{\hat{w}}\subseteq \Lambda$ denote the extensions $\phi$ such that $\phi(\hat{w})=a$, noting by the Pigeonhole Principle there must exist some $\hat{v}\in V(H)\setminus V(H')$ such that $|\Lambda^{\hat{v}}|\geq \eta \beta(H')/v(H)$. Note that every $\phi\in\Lambda^{\hat{v}}$ has a restriction in $\Lambda^{\hat{v}}_{H-\hat{v}}$, and each restriction uniquely extends back to $\Lambda^{\hat{v}}$ since the extension must send $\hat{v}\mapsto a$, so $|\Lambda^{\hat{v}}_{H-\hat{v}}|=|\Lambda^{\hat{v}}|$. Furthermore, since $\Lambda^{\hat{v}}_{H-\hat{v}}\subseteq \Phi_{H-\hat{v}}$, we have that $\deg_{\Phi}(\xi)\geq \beta(H-\hat{v})$ for all $\xi\in \Lambda^{\hat{v}}_{H-\hat{v}}$, and in particular, $\deg_{\Phi-a}(\xi)\geq \beta(H-\hat{v})-1$.
		
		Now, we can lower bound the degree of $\psi$ in $\Phi-a$. We first choose some $\xi\in \Lambda^{\hat{v}}_{H-\hat{v}}\sub (\Phi-a)_{H-\hat{v}}$, then choose an extension $\phi$ of $\xi$ in $\Phi-a$. Note that each $\phi$ we construct in this way is distinct by our previous discussion and is an extension of $\psi$ in $\Phi-a$ by the definition of $\Lambda$, implying that 
		\[
		\deg_{\Phi-a}(\psi)\geq |\Lambda^{\hat{v}}_{H-\hat{v}}|\cdot \deg_{\Phi-a}(\xi)\geq \frac{\eta \beta(H')}{v(H)}(\beta(H-\hat{v})-1)\geq \beta(H')>(1-\eta)\beta(H'),
		\]
		where the second to last inequlaity used the hypothesis on $\beta$ from the lemma.
		We conclude that, $\Phi-a$ is a $(H,(1-\eta)\beta)$-weak builder, completing the proof.
	\end{proof}
	
	We next aim to show that strong builders imply the existence of many copies of a given $H$-tree.  For this we need the following definitions, where here we recall the definition of a sequence $(h_1,j_2,\ldots,h_t)$ being a witness of an $H$-tree made at the end of \Cref{definition H-tree}. 
	\begin{definition}
		We say that a witness $(h_1,j_2,h_2,j_3,\dots,h_t)$ of an $H$-tree $F$ is of \textbf{length} $t$ and \textbf{type} $\gamma$ where $\gamma:2^H\sm \{H\}\to\{0,1,\ldots,t-1\}$ is the function with $\gamma(H')$ defined to be the number of $i$ such that 
		\[
		H'=H[h_i^{-1}(V(H_i)\cap V(H_{j_i}))]
		\] 
		as $i$ ranges from $2$ to $t$. 
		
		Given a collection $\Phi$ of monomorphisms of $H$ into some graph $G$, we say that $(\phi_1,\phi_2,\dots,\phi_t)$ with $\phi_i\in \Phi$ for all $i$ is a \textbf{lift} of the witness $(h_1,j_2,h_2,j_3,\dots,h_t)$ into $\Phi$ if for all $\hat{u},\hat{v}\in V(H)$ we have $\phi_i(\hat{u})=\phi_j(\hat{v})$ if and only if $h_i(\hat{u})=h_j(\hat{v})$. The \textbf{image} of the lift $(\phi_1,\phi_2,\dots,\phi_t)$ in $G$ is the set $\bigcup_{i=1}^t \phi_i(V(H))$.
	\end{definition}

	As a point of reference, the definition of $H$-trees can be shown to imply that $h_i(\hat{u})=h_j(\hat{v})$ only happens if $\hat{u}=\hat{v}$.  However, it will be convenient to allow for the a priori possibility of having $\hat{u}\ne \hat{v}$ in the definition of lifts of witnesses.  In particular, we use this in the following result which says that lifts imply the existence of copies of $F$.
	
	\begin{lemma}\label{lemma copy implies mono}
		Let $G$ be a graph and $\Phi\subseteq \mathrm{Mon}(F,G)$. Let $F$ be a $H$-tree with defining copies $\{H_1,H_2,\dots,H_t\}$ and witness $W:=(h_1,j_2,\ldots,h_t)$. If $P:=(\phi_1,\ldots,\phi_t)$ is a lift of the witness $W$ into $\Phi$, then there exists a monomorphism from $F$ to $G$ with image $\displaystyle\bigcup_{i=1}^t \phi_i(V(H))$.
	\end{lemma}
	
	\begin{proof}
		Let $\rho:V(F)\to V(G)$ be defined by setting $\rho(u)=\phi_i(h_i^{-1}(u))$ for each $u\in V(H_i)$ and $i\in [t]$. Since $V(F)=\bigcup_{i=1}^t V(H_i)$, every vertex $u\in V(F)$ is assigned at least one image by $\rho$. The map $\rho$ is well defined since if we have $u\in V(H_i)\cap V(H_j)$ for some $i\neq j$, then there exists some $\hat{v},\hat{w}\in V(H)$ such that $u=h_i(\hat{v})=h_j(\hat{w})$, but then by the definition of a lift into $\Phi$ we have $\phi_i(\hat{v})=\phi_j(\hat{w})$. In particular $\phi_i(h_i^{-1}(u))=\phi_j(h_j^{-1}(u))$. 
		
		The map $\rho$ has image $\bigcup_{i=1}^t \phi_i(V(H))$ by definition, and in fact $\rho$ maps injectively since if $u\in V(H_i)$ and $v\in V(H_j)$ with $\rho(u)=\rho(v)$, then $\phi_i(h_i^{-1}(u))=\phi_j(h_j^{-1}(v))$, and thus by the definition of a lift into $\Phi$ we have \[u=h_i(h_i^{-1}(u))=h_j(h_j^{-1}(v))=v.\]  Finally, we note that $\rho|_{V(H_i)}=\phi_i\circ h_i^{-1}$ is a composition of homomorphisms, and thus maps edges of $H_i$ to edges of $G$. Since every edge of $F$ is in some $H_i$, the map $\rho$ is a homomorphism.
	\end{proof}
	We now prove the main fact we need about lifts of witnesses.
	
	\begin{lemma}\label{lemma counting trees in a strong builder}
		Let $F$ be an $s$-vertex $H$-tree with a witness $(h_1,j_1,\dots,h_t)$ of type $\gamma$, and let $\Phi$ be an $(H,\beta,s')$-strong builder for some graph $G$ with $s'\geq s$. Then the number of lifts of $(h_1,j_1,\dots,h_t)$ into $\Phi$ is at least
		\[
		|\Phi|\cdot\prod_{H'\in 2^H} \beta(H')^{\gamma(H')}.
		\]
	\end{lemma}
	
	\begin{proof}
		We will proceed by induction on $t$. For $t=1$, the only $H$-tree with a witness of length $1$ is $H$ itself. Given a witness $W=(h_1)$ of length $1$, we note that the type $\gamma$ of $W$ is the constant $0$ function. Furthermore, for any $\phi\in \Phi$, $(\phi)$ is a lift of $(h_1)$ since both $\phi$ and $h_1$ are injective. Thus, there are at least $|\Phi|$ lifts of $(h_1)$.
		
		Now assume $t\geq 2$ and let $F$ be an $H$-tree with $|V(F)|\leq s'$, with defining copies $H_1,H_2,\dots,H_t$, and with witness $(h_1,j_1,\dots,h_t)$ of type $\gamma$. Define $H^*:=H[h_t^{-1}(V(H_t)\cap V(H_{j_t}))]$. Let $F'$ be the $H$-tree with defining copies $H_1,H_2,\dots,H_{t-1}$, and note that $V(F')\leq V(F)\leq s'$, and that $(h_1,j_1,\dots,h_{t-1})$ is a witness of $F'$, say with type $\gamma'$, and note that
		\[
		\gamma'(H')=\begin{cases}
			\gamma(H')-1&\text{ if }H'=H^*,\\
			\gamma(H')&\text{ otherwise}.
		\end{cases}
		\]
		Now, fix a lift $(\phi_1,\dots,\phi_{t-1})$ of $(h_1,j_1,\dots,h_{t-1})$ into $\Phi$. We will show there are many $\phi\in \Phi$ such that $(\phi_1,\dots,\phi_{t-1},\phi)$ is a lift of $(h_1,j_1,\dots,h_t)$. Indeed, if we set 
		\[
		S:=\bigcup_{i=1}^{t-1}\phi_i(V(H))\setminus \phi_{j_t}(V(H^*)),
		\]
		and note that $|S|\leq |V(F')|\leq s'$, then we have that $\Phi-S$ is a $(H,\beta)$-weak builder. Furthermore, $\phi_{j_t}|_{V(H^*)}\in \Phi_{H^*}-S\subseteq (\Phi-S)_{H^*}$ (since $\Phi$ is a strong-builder), and so there are $\beta(H^*)$ extensions $\phi$ of $\phi_{j_t}|_{V(H^*)}$ which are disjoint from $S$. It is straightfoward from the definition of a lift and witness that $(\phi_1,\phi_2,\dots,\phi_{t-1},\phi)$ is a lift of $(h_1,j_1,\dots,h_t)$. By induction, this gives us
		\[
		|\Phi|\cdot\prod_{H'\in 2^H} \beta(H')^{\gamma'(H')}\cdot \beta(H^*)=|\Phi|\cdot\prod_{H'\in 2^H} \beta(H')^{\gamma(H')}
		\]
		lifts of $(h_1,j_1,\dots,h_t)$.
	\end{proof}
	
	The last thing we need is the following lemma relating the type of an $H$-tree $F$ with $v(F)$ and $e(F)$.
	\begin{lemma}\label{lemma counting v and e in a tree}
		Let $F$ be an $H$-tree with a witness $(h_1,j_2,\dots,h_t)$ of type $\gamma$. Then
		\[
		v(H)+\sum_{H'\in 2^H} \gamma(H')(v(H)-v(H'))=v(F)
		\]
		and
		\[
		e(H)+\sum_{H'\in 2^H} \gamma(H')(e(H)-e(H'))=e(F).
		\]
	\end{lemma}
	
	\begin{proof}
		We write
		\begin{align*}
			v(F)=\left|\bigcup_{i=1}^t h_i(V(H))\right|&=\sum_{i=1}^t \left|h_i(V(H))\setminus\left(\bigcup_{k<i} h_k(V(H))\right)\right|\\
			&=\sum_{i=1}^t\left( |h_i(V(H))|-\left| h_i(V(H))\cap \left(\bigcup_{k<i} h_k(V(H))\right)\right|\right)\\
			&=|h_1(V(H))|+\sum_{i=2}^t \left(|h_i(V(H))|-\left| h_i(V(H))\cap h_{j_i}(V(H))\right|\right)\\
			&=v(H)+\sum_{i=2}^t\left( v(H)-\left| h_i(V(H))\cap h_{j_i}(V(H))\right|\right),
		\end{align*}
		then for $H'\in 2^H$ as $i$ ranges from $2$ to $t$ in the above sum, we have exactly $\gamma(H')$ indices for which $\left| h_i(V(H))\cap h_{j_i}(V(H))\right|=|V(H')|$, so we may write
		\[
		v(H)+\sum_{i=2}^t \left(v(H)-| h_i(V(H))\cap h_{j_i}(V(H))|\right)=v(H)+\sum_{H'\in 2^H} \gamma(H')(v(H)-v(H')).
		\]
		Let us now focus on the edges of $F$. For each $i\in [t]$, write $E_i:=\{h_i(\hat{u})h_i(\hat{v})\mid \hat{u}\hat{v}\in E(H)\}$. Then, similarly as in the case for vertices, we can write
		\begin{align*}
			e(F)&=\sum_{i=1}^t\left( |E_i|-\left| E_i\cap \left(\bigcup_{k<i} E_k\right)\right|\right)\\
			&=e(H)+\sum_{i=2}^t\left( e(H)-\left|E_i\cap E_{j_i}\right|\right)\\
			&=e(H)+\sum_{H'\in 2^H}\gamma(H')\left(e(H)-e(H')\right).
		\end{align*}
	\end{proof}
	
	We can now prove Theorem~\ref{theorem d large tree implies dadmissible}.
	
	\begin{proof}[Proof of Theorem~\ref{theorem d large tree implies dadmissible}]
		The case $H=K_1$ follows from \Cref{observation K1 is easy}, so from now on we assume $H$ is a connected graph on at least 2 vertices.

		Let $\pow\in\mathbb{N}$, and set $p:= Cn^{-1/d}$ for a constant $C=C(\pow)>0$ that we will specify later. Let $\mathcal{R}$ denote the collection of components of $F[R]$. For each $H'\in (2^{H}\setminus \{H\})\cup \mathcal{R}$,  let $(F_{H',1},R_{H',1}),\dots,(F_{H',t_{H'}},R_{H',t_{H'}})$ be a collection of rooted graphs of rooted density at least $d$ for which $H'$ is $d$-admissible with respect to (which exists by hypothesis of the theorem).  This implies there exists some constant $C_{H'}=C_{H'}(\pow)$ such that
		\[
		\mathrm{ex}\left(n,H',\bigcup_{i=1}^{t_{H'}}(F_{H',i},R_{H',i})^{\pow}\right)\leq C_{H'}\cdot n^{v(H')-e(H')/d}.
		\]
		Let $\mathcal{F}$ denote the entire collection of graphs $(F_{H',i},R_{H',i})^{\pow}$ for all $H'\in (2^{H}\setminus \{H\})\cup \mathcal{R}$ and $i\in [t_{H'}]$. Let $\mathcal{A}\subseteq 2^{H}$ consist of all the proper induced subgraphs $H'\subsetneq H$ such that $E(H')\neq \emptyset$. Then, as long as $C\geq \max_{H'\in \mathcal{A}}\{\mathrm{aut}(H')\cdot C_{H'}\}$, we have that for any $H'\in\mathcal{A}$,
		\begin{equation}\label{equation nonempty subgraphs are bounded}
			\mathrm{aut}(H')\cdot \mathrm{ex}(n,H',\mathcal{F})\leq Cn^{v(H')-e(H')/d}\leq C^{e(H)}n^{v(H')-e(H')/d}=n^{v(H')}p^{e(H')}.
		\end{equation}
		Furthermore, for $H'\in 2^{H}\setminus\mathcal{A}$ with $H'\neq H$, we have $E(H')=\emptyset$, so we trivially have 
		\begin{equation}\label{equation empty subgraphs are bounded}
			\mathrm{aut}(H')\cdot \mathrm{ex}(n,H',\mathcal{F})\leq n^{v(H')}=n^{v(H')}p^{e(H')}.
		\end{equation}
		
		With this in hand, let $G$ be a graph that is $\mathcal{F}$-free which has
		\[
		\mon(H,G)\ge \left(2^{v(H)+1}C^{e(H)}\right)n^{v(H)-e(H)/d}=2^{v(H)+1}n^{v(H)}p^{e(H)}.
		\]
		We will show that $G$ contains a graph in $(F,R)^{\pow}$ as a subgraph, which in view of \Cref{observation monomorphisms same as copies} will imply that
		\[
		\mathrm{ex}\left(n,H,\mathcal{F}\cup\left\{(F,R)^{\pow}\right\}\right)=O_{\pow}(n^{v(H)-e(H)/d}),
		\]
		completing the proof. We will do this by showing that $G$ contains few copies of $F[R]$ and many copies of $F$, with this latter goal requiring us to show the existence of a strong builder as follows.
		
		For each $H'\in 2^H\setminus\{H\}$, let $\beta(H'):=n^{v(H)-v(H')}p^{e(H)-e(H')}$. By~\eqref{equation nonempty subgraphs are bounded} and~\eqref{equation empty subgraphs are bounded}, we have by \Cref{observation monomorphisms same as copies} that $\mon(H',G)\le  n^{v(H')}p^{e(H')}$ for each $H'\in 2^H\setminus\{H\}$, so by Lemma~\ref{lemma weak builder} applied to $\Mon(H,G)$, there exists a $(H,2\beta)$-weak builder $\Phi\subseteq \mathrm{Mon}(H,G)$. Furthermore, for each $\hat{w}\in V(H)$, we note that
		\[
		\beta(H-\hat{w})=np^{\deg(\hat{w})}=C^{\deg(\hat{w})}n^{1-\deg(\hat{w})/d}\geq C,
		\]
		where in this last step we used the assumption $d\ge \Delta(H)$ and $\deg(\hat{w})\ge 1$ due to connectivity. In particular, as long as $C$ is chosen such that $C\geq \frac{v(H)}{1-2^{-1/v(F)}}+s$, we have $\beta(H-\hat{w})\geq \frac{v(H)}{1-2^{-1/v(F)}}+s$ for all $\hat{w}\in V(H)$. Thus, by Lemma~\ref{lemma strong builder} applied with $s:=v(F)$, we actually have that $\Phi$ is a $(H,\beta,v(F))$-strong builder.

		We now show that $G$ contains many copies of $F$.  By definition of $F$ being an $H$-tree it has some witness $(h_1,j_2,\ldots,h_t)$ which is of some type $\gamma$.  By Lemma~\ref{lemma counting trees in a strong builder},  the number of lifts of this witness in $\Phi$ is at least 
		\begin{equation}\label{equation copies of T in Phi}
			|\Phi|\cdot\prod_{H'\in 2^H} \beta(H')^{\gamma(H')}\geq n^{v(H)}p^{e(H)}\prod_{H'\in 2^H} \left(n^{v(H)-v(H')}p^{e(H)-e(H')}\right)^{\gamma(H')}.
		\end{equation}
		Expanding the above product and collecting the powers of $n$ (not including $p$), then applying Lemma~\ref{lemma counting v and e in a tree} we find the power of $n$ in~\eqref{equation copies of T in Phi} is
		\[
		v(H)+\sum_{H'\in 2^H}d_H(v(H)-V(H'))=v(F),
		\]
		and similarly, the power of $p$ is
		\[
		e(H)+\sum_{H'\in 2^H}\gamma(H')(e(H)-e(H'))=e(F).
		\]
		
		Thus, we find we have at least $n^{v(F)}p^{e(F)}$ lifts of this witness into $\Phi$.  Moreover, \Cref{lemma copy implies mono} implies each of these lifts $(\phi_1,\phi_2,\dots,\phi_t)$ gives us a monomorphism of $F$ into $G$ with image $\bigcup_{i=1}^t\phi_i(V(H))$.  As the number of lifts of the witness $(h_1,j_2,\ldots,h_t)$ with a given image of size $v(F)$ is trivially at most $v(F)^{tv(H)}$, we in total conclude that
		\[
		\mon(F,G)\geq \frac{n^{v(F)}p^{e(F)}}{v(F)^{tv(H)}}.
		\]
		
		We next count $\mathrm{mon}(F[R],G)$. By~\eqref{equation nonempty subgraphs are bounded} and~\eqref{equation empty subgraphs are bounded} (along with Observation~\ref{observation monomorphisms same as copies}), for each $R_i\in\mathcal{R}$, we have that $\mathrm{mon}(R_i,G)\leq n^{v(R_i)}p^{e(R_i)}$, so in total the number of monomorphisms of $F[R]$ into $G$ is at most
		\[
		\prod_{R_i\in\mathcal{R}} n^{v(R_i)}p^{e(R_i)}=n^{v(F[R])}p^{e(F[R])}.
		\]
		Finally, if we have, say $k$ monomorphisms of $F$ into $G$ which all map $F[R]$ onto some fixed set $S$ with $|S|=v(F[R])$, then there must be at least $k/v(F[R])!$ monomorphisms of $F$ into $G$ which map $F[R]$ onto $S$ such that the restriction to $S$ is the same. Putting everything together, by the Pigeonhole Principle, there must exist at least
		\[
		M':=\frac{n^{v(F)}p^{e(F)}}{(v(F[R])!\cdot v(F)^{tv(H)}n^{v(F[R])}p^{e(F[R])}}=\frac{n^{v(F)-v(F[R])}p^{e(F)-e(F[R])}}{(v(F[R])!\cdot v(F)^{tv(H)}}
		\]
		monomorphisms of $F$ into $G$ which all agree on $F[R]$. By Observation~\ref{observation monomorphisms same as copies}, this gives us that there are at least $M:=M'/\mathrm{aut}(F)$ distinct copies of $F$ which agree on $F[R]$ in $G$, i.e. $G$ contains a copy of some element in $(F,R)^M$. Recalling that $p=Cn^{-1/d}$ and $\frac{v(F)-v(F[R])}{e(F)-e(F[R])}=d$, we find
		\[
		M=\frac{C^{e(F)-e(F[R])}}{\mathrm{aut}(F)\cdot (v(F[R])!\cdot v(F)^{tv(H)}}.
		\]
		Thus, as long as we choose $C\geq \left(\pow\cdot \mathrm{aut}(F)\cdot (v(F[R])!\cdot v(F)^{tv(H)}\right)^{1/(e(F)-e(F[R]))}$, we have a copy of $(F,R)^{\pow}$ in $G$. It is straightforward to see that $C=C(\pow)$ can be chosen large enough to satisfy all the required inequalities, completing the proof.
	\end{proof}

	\subsection{Building \texorpdfstring{$H$}{H}-Trees}\label{subsection building H-trees}
	In this subsection we build  $H$-trees in order to apply \Cref{theorem d large tree implies dadmissible} to prove \Cref{theorem:maxDegreeGeneral} and \Cref{theorem:maxDegreeTrees}.  We will primarily build our $H$-trees by iteratively adding new copies of $H$ to a given $H$-tree $F$, for which the following will be helpful.
	
	\begin{lemma}\label{lemma iteratively building H-trees}
		If $F$ is an $H$-tree with defining copies $\{H_1,\ldots,H_{t-1}\}$ and if $F'$ is the union of $F$ and a copy $H_{t}$ of $H$ with $H_t\not\in \{H_1,\dots,H_{t-1}\}$, then $F'$ is an $H$-tree provided there exists some $j_{t}$ such that 
		\[V(H_{t})\cap V(F)=V(H_{t})\cap V(H_{j_{t}}),\]
		and such that there is an isomorphism $\phi$ from $H_{j_t}$ to $H_t$ such that $\phi(u)=u$ for all $u\in V(H_{t})\cap V(H_{j_{t}})$.
	\end{lemma}
	
	\begin{proof}
		Since $F$ is an $H$-tree, there exists isomorphisms $h_i:V(H)\to V(H_i)$ for all $i\in [t-1]$. We define $h_t:V(H)\to V(H_t)$ by the composition $h_t:=\phi\circ h_{j_t}$. With these isomorphisms, it is straightforward to verify that $F'$ is an $H$-tree with defining copies $H_1,\dots,H_t$.
	\end{proof}
	
	In order to construct \textit{balanced} $H$-trees we will use the following sufficient criteria, which follows immediately from the definitions of densities and balancedness.
	
	\begin{observation}\label{observation easy balanced condition}
		Let $(F_1,R_1)$ and $(F_2,R_2)$ be a pair of rooted graphs such that there exists a function $g:V(F_1)\setminus R_1\to V(F_2)\setminus R_2$ such that $g$ is an isomorphism from $F_1-R_1$ to $F_2-R_2$ and such that $\deg_{F_1}(x)=\deg_{F_2}(g(x))$ for all $x\in V(F_1)\setminus R_1$. Then the rooted density of $(F_1,R_1)$ and $(F_2,R_2)$ are the same, and $(F_1,R_1)$ is balanced if and only if $(F_2,R_2)$ is balanced.
	\end{observation}
	We will always apply this result with $(F_2,R_2)$ the balanced tree $(T,R_T)$ guaranteed by \Cref{lem:BukhConlon} whose non-rooted vertices form a path $u_1\cdots u_a$.  In view of this, it  will suffice for us to construct a rooted $H$-tree $(F,R)$ such that $F-R$ forms a path $u'_1\cdots u'_a$ with $\deg_F(u'_i)=\deg_T(u_i)$ for all $i$.  
	
	Our remaining proofs follow the strategy outlined above in roughly two steps: we first build some $H$-tree $(F',R')$ such that $F'-R'$ forms a path $u'_1\cdots u'_a$ with $\deg_{F'}(u'_i)\le \deg_T(u_i)$ for all $i$, after which we add additional copies of $H$ to $F'$ in order to increase each of these degrees until they equal $\deg_T(u_i)$ while maintaining that the components of $F'-R'$ are well-behaved.  Moreover, our constructions for $(F',R')$ will start by specifying some edge $e\in E(H)$ and then building $F'$ as a union of copies of $H$ such that each edge $u_i'u_{i+1}'$ in the path $F'-R'$ plays the role of $e$ in each copy of $H$.  This procedure will be simplest when $e$ is a cut edge, i.e.\ an edge such that $H-e$ is disconnected.
	
	\begin{lemma}\label{lemma cut edge dadmissible}
		Let $H$ be a connected graph of maximum degree at most $\Delta$ which has a cut-edge $e=\hat{x}\hat{y}$.  If $d\ge \Delta$ is rational and if every proper induced subgraph of $H$ is $d$-admissible, then $H$ is $d$-admissible.
	\end{lemma}
	
	\begin{proof}

		Throughout this proof, if $H_i$ is a copy of $H$ and if $H'\sub H$ is an induced subgraph, then we write $H'_i$ to denote the subgraph of $H_i$ corresponding to $H'$.  
		Let $H^{\hat{x}}$ and $H^{\hat{y}}$ denote the components of $H-e$ containing $\hat{x}$ and $\hat{y}$ respectively.
		
		By Theorem~\ref{theorem d large tree implies dadmissible}, it suffices to prove there exists a balanced $H$-tree of rooted density $d$ whose root set contains only components in $2^H\setminus\{H\}$, and we will try to design such an $H$-tree to resemble the balanced trees from \Cref{lem:BukhConlon}.  To this end, let $a,b$ be positive integers with $a$ even such that $d=b/a$, and let $(T,R_T)$ be the balanced rooted tree of rooted density $d$ with $a$ unrooted vertices $u_1,u_2,\dots,u_a$ guaranteed by \Cref{lem:BukhConlon}.  For technical reasons, if $d=\Delta$ we will further enforce that $a=2$ and $b=2\Delta$.
		
		To build our $H$-tree, we start with $H_1$ being an arbitrary copy of $H$.  Iteratively, given that we have defined $H_1,\ldots,H_{i-1}$ for some even $2\le i\le a-1$, we define $H_i$ to be a copy which has $H^{\hat{y}}_i=H^{\hat{y}}_{i-1}$ and which is disjoint from every other vertex of the $H_j$ copies.  Similarly if $2\le i\le a-1$ is odd we define $H_i$ to be a copy with $H^{\hat{x}}_i=H^{\hat{x}}_{i-1}$ and which is disjoint from every other vertex of the other $H_j$. Let $F'$ be the union of these $H_1,\ldots,H_{a-1}$ copies. 
		
		For each odd integer $1\le i\le a-1$, let $x_i$ denote the vertex of $H_i$ which plays the role of $\hat{x}$ and also let $y_{i+1}$ denote the vertex of $H_{i}$ which plays the role of $\hat{y}$.  With this naming convention, we see by construction that each $x_i$ with $i\ne 1$ plays the role of $\hat{x}$ in both $H_i$ and $H_{i-1}$, and similarly each $y_i$ with $i\ne a$ plays the role of $\hat{y}$ in both $H_i$ and $H_{i-1}$. Let $U\sub V(F')$ (with ``$U$'' standing for ``unrooted''') denote the set $U:=\{x_1,y_2,\ldots,x_{a-1},y_a\}$, and define $g:U \to V(T)\setminus R_T$ by $g(x_i)=u_i$ for odd $i$ and $g(y_i)=u_i$ for even $i$.

		\begin{figure}
			\begin{center}
				\begin{subfigure}[t]{0.2\textwidth}
					\begin{tikzpicture}
						\draw[thick] (0,0)--(-0.5,-1)--(.5,-1)--(0,0)--(1.5,0)--(2,-0.55)--(1.8,-1.2);
						
						\draw[dashed, rounded corners] (-.7, 0.6) rectangle (.7,-1.2);
						\draw[dashed, rounded corners] (1.3, 0.6) rectangle (2.2,-1.4);
						
						\node (0) at (0,0) {};
						\node (1) at (1.5,0) {};
						\node (2) at (-0.5,-1) {};
						\node (3) at (0.5,-1) {};
						\node (4) at (2,-0.55) {};
						\node (5) at (1.8,-1.2) {};
						
						\node (6) at (0,0.85) {$H_{\hat{x}}$};
						\node (7) at (1.75,0.85) {$H_{\hat{y}}$};
						
						\fill (0) circle (0.1) node [above] {$\hat{x}$};
						\fill (1) circle (0.1) node [above] {$\hat{y}$};
						\fill (2) circle (0.1) node [below] {};
						\fill (3) circle (0.1) node [below left] {};
						\fill (4) circle (0.1) node [below] {};
						\fill (5) circle (0.1) node [below] {};
					\end{tikzpicture}
					\caption{The graph $H$. Subgraphs $H_{\hat{x}}$ and $H_{\hat{y}}$ are circled.}\label{figure tree proof H}
				\end{subfigure}\begin{subfigure}[t]{0.4\textwidth}
					\begin{tikzpicture}
						\draw[thick] (0,0.1)--(-0.35,-.8)--(.35,-.8)--(0,0.1)--(1,0.1)--(1.2,-0.65)--(1.1,-1.3);
						\draw[thick] (2,0.1)--(1.65,-.8)--(2.35,-.8)--(2,0.1)--(3,0.1)--(3.2,-0.65)--(3.1,-1.3);
						\draw[thick] (4,0.1)--(3.65,-.8)--(4.35,-.8)--(4,0.1)--(5,0.1)--(5.2,-0.65)--(5.1,-1.3);
						\draw[thick] (0,0.1)--(5,0.1);
						\draw[thick,dashed] (-0.65,-0.3)--(5.65,-0.3);
						
						\node (0) at (0,0.1) {};
						\node (1) at (1,0.1) {};
						\node (2) at (2,0.1) {};
						\node (3) at (3,0.1) {};
						\node (4) at (4,0.1) {};
						\node (5) at (5,0.1) {};
						\node (6) at (1.65,-.8) {};
						\node (7) at (2.35,-.8) {};
						\node (8) at (3.2,-0.65) {};
						\node (9) at (3.1,-1.3) {};
						\node (10) at (-0.35,-.8) {};
						\node (11) at (.35,-.8) {};
						\node (12) at (1.2,-0.65) {};
						\node (13) at (1.1,-1.3) {};
						\node (14) at (3.65,-.8) {};
						\node (15) at (4.35,-.8) {};
						\node (16) at (5.2,-0.65) {};
						\node (17) at (5.1,-1.3) {};
						
						\fill (0) circle (0.06) node [above] {$x_1$};
						\fill (1) circle (0.06) node [above] {$y_2$};
						\fill (2) circle (0.06) node [above] {$x_3$};
						\fill (3) circle (0.06) node [above] {$y_4$};
						\fill (4) circle (0.06) node [above] {$x_5$};
						\fill (5) circle (0.06) node [above] {$y_6$};
						\fill (6) circle (0.06) node [above] {};
						\fill (7) circle (0.06) node [above] {};
						\fill (8) circle (0.06) node [above] {};
						\fill (9) circle (0.06) node [above] {};
						\fill (10) circle (0.06) node [above] {};
						\fill (11) circle (0.06) node [above] {};
						\fill (12) circle (0.06) node [above] {};
						\fill (13) circle (0.06) node [above] {};
						\fill (14) circle (0.06) node [above] {};
						\fill (15) circle (0.06) node [above] {};
						\fill (16) circle (0.06) node [above] {};
						\fill (17) circle (0.06) node [above] {};
					\end{tikzpicture}
					\caption{The graph $F'=F_0$.}\label{figure tree proof F'=F_0}
				\end{subfigure}\begin{subfigure}[t]{0.4\textwidth}
					\begin{tikzpicture}
						\draw[thick] (0,0.1)--(-0.35,-.8)--(.35,-.8)--(0,0.1)--(1,0.1)--(1.2,-0.65)--(1.1,-1.3);
						\draw[thick] (2,0.1)--(1.55,-.8)--(2.2,-.8)--(2,0.1)--(3,0.1);
						\draw[thick] (4,0.1)--(3.65,-.8)--(4.35,-.8)--(4,0.1)--(5,0.1)--(5.2,-0.65)--(5.1,-1.3);
						\draw[thick] (0,0.1)--(5,0.1);
						\draw[thick,dashed] (-0.65,-0.3)--(5.65,-0.3);
						\draw[thick,red] (2.74,-.45)--(2.5,-0.95)--(3.05,-0.95)--(2.74,-0.45)--(3,0.1)--(3.35,-0.65)--(3.275,-1.3);
						
						\node (0) at (0,0.1) {};
						\node (1) at (1,0.1) {};
						\node (2) at (2,0.1) {};
						\node (3) at (3,0.1) {};
						\node (4) at (4,0.1) {};
						\node (5) at (5,0.1) {};
						\node (6) at (1.55,-.8) {};
						\node (7) at (2.2,-.8) {};
						\node (8) at (3.35,-0.65) {};
						\node (9) at (3.275,-1.3) {};
						\node (10) at (-0.35,-.8) {};
						\node (11) at (.35,-.8) {};
						\node (12) at (1.2,-0.65) {};
						\node (13) at (1.1,-1.3) {};
						\node (14) at (3.65,-.8) {};
						\node (15) at (4.35,-.8) {};
						\node (16) at (5.2,-0.65) {};
						\node (17) at (5.1,-1.3) {};
						\node (18) at (2.74,-0.45) {};
						\node (19) at (2.5,-0.95) {};
						\node (20) at (3.05,-0.95) {};
						
						\fill (0) circle (0.06) node [above] {$x_1$};
						\fill (1) circle (0.06) node [above] {$y_2$};
						\fill (2) circle (0.06) node [above] {$x_3$};
						\fill (3) circle (0.06) node [above] {$y_4$};
						\fill (4) circle (0.06) node [above] {$x_5$};
						\fill (5) circle (0.06) node [above] {$y_6$};
						\fill (6) circle (0.06) node [above] {};
						\fill (7) circle (0.06) node [above] {};
						\fill (8) circle (0.06) node [above] {};
						\fill (9) circle (0.06) node [above] {};
						\fill (10) circle (0.06) node [above] {};
						\fill (11) circle (0.06) node [above] {};
						\fill (12) circle (0.06) node [above] {};
						\fill (13) circle (0.06) node [above] {};
						\fill (14) circle (0.06) node [above] {};
						\fill (15) circle (0.06) node [above] {};
						\fill (16) circle (0.06) node [above] {};
						\fill (17) circle (0.06) node [above] {};
						\fill[color=blue] (18) circle (0.06) node [above] {};
						\fill[color=blue] (19) circle (0.06) node [above] {};
						\fill[color=blue] (20) circle (0.06) node [above] {};
					\end{tikzpicture}
					\caption{The graph $F_1$. The edges of $H_{a-1+1}$ are red, vertices in $V(H_{a-1+i})\setminus V(F_0)$ are blue.}\label{figure tree proof F_1}
				\end{subfigure}
				\caption{The graphs $F'=F_0$ and $F_1$ for the example graph $H$ given in Figure~\ref{figure tree proof H} in the case $d>\Delta$, $a=6$.  Adding $H_{a-1+1}$ to $F_0$ to form $F_1$ increases the degree of $y_4$ by $1$ while maintaining properties~\ref{property tree proof H tree},\ref{property tree proof connected components} and~\ref{property tree proof degree of spine}. The set $U$ consists of the vertices above the dashed lines; all vertices below the dashed lines end up in the root in the final rooted $H$-tree $F$.}\label{figure tree proof main figure}
			\end{center}
		\end{figure}
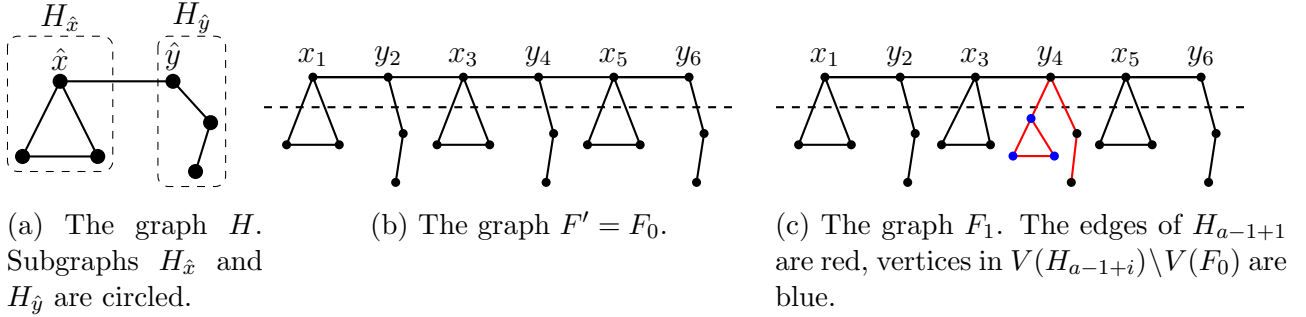

		We will now inductively define a sequence of graphs $F_0\subseteq F_1\subseteq \dots$ with $F'=F_0$ 
		together with a sequence of copies $H_a,H_{a+1},\ldots$ of $H$ that satisfy the following properties for all $i\geq 0$:
		\begin{enumerate}[label=(T\arabic*)]
			\item $F_i$ is an $H$-tree with defining copies $\{H_1,H_2,\dots,H_{a-1+i}\}$,\label{property tree proof H tree}
			\item The connected components of $F_i-U$ are all proper induced subgraphs of $H$,\label{property tree proof connected components}
			\item The map $g:U\to V(T)\setminus R_T$ defined by $g(x_i)=u_i$ for odd $i$ and $g(y_i)=u_i$ for even $i$ is an isomorphism from $F_i-U$ to $T-R_T$ and satisfies
			\begin{equation}\label{equation tree proof degree inequality}
				\deg_{F_i}(u)\le \deg_T(g(u))
			\end{equation}
			for all $u\in U$.\label{property tree proof degree of spine}
		\end{enumerate}
		We will construct this sequence such that the final graph in the sequence will satisfy~\eqref{equation tree proof degree inequality} with equality for all $u\in U$.
		
		Before we describe the sequence of graphs, let us show that $F_0=F'$ satisfies~\ref{property tree proof H tree},~\ref{property tree proof connected components} and~\ref{property tree proof degree of spine}. By repeated applications of Lemma~\ref{lemma iteratively building H-trees}, we note that $F'$ is an $H$-tree with defining copies $\{H_1,H_2,\dots,H_{a-1}\}$, so \ref{property tree proof H tree} is satisfied. Similarly, it is not difficult to see that $F'-U$ is equal to the disjoint union of copies of $H-\hat{x}-\hat{y}$, giving~\ref{property tree proof connected components}. By construction, $F'[U]$ is the path $x_1y_2\dots x_{a-1}y_a$, so $g$ is an isomorphism from $F'-(V(F')\setminus U)$ to $T-R_T$.  Now consider some $u\in U$.  If $d=\Delta$, then since in this case $a=2$, we have $F'\cong H$, so  $\deg_{F'}(u)\leq \Delta=d\leq \deg_T(g(u))$ with us using \Cref{lem:BukhConlon} for the bound on $\deg_T(g(u))$ here.  If instead $d>\Delta$, then we have $\deg_{F'}(u)\leq \Delta+1\leq \lceil d\rceil\leq \deg_T(g(u))$, so $F'$ satisfies~\ref{property tree proof degree of spine}.
		
		Inductively, let us assume that for some $i\geq 1$, we have that $F_{i-1}$ satisfies~\ref{property tree proof H tree},~\ref{property tree proof connected components} and~\ref{property tree proof degree of spine}, and has $F'$ as a subgraph. Let us further assume that there exists $x_j\in U$ such that $\deg_{F_{i-1}}(x_j)<\deg_T(g(x_j))$ (the case where $\deg_{F_{i-1}}(y_j)<\deg_T(g(y_j))$ is symmetric). Define $H_{a-1+i}$ to be a copy of $H$ such that  $H_{a-1+i}^{\hat{x}}=H_j^{\hat{x}}$ and such that it is otherwise disjoint from $F_{i-1}$, and let $F_i:=F_{i-1}\cup H_{a-1+i}$. See Figure~\ref{figure tree proof main figure} for an example of building $F_1$ from $F_0$. By~\Cref{lemma iteratively building H-trees}, $F_i$ is a $H$-tree with defining copies $\{H_1,\dots,H_{a-1+i}\}$, so \ref{property tree proof H tree} holds. Furthermore, $F_i-U\cong (F_{i-1}-U)\sqcup (H-V(H^{\hat{x}}))$, giving \ref{property tree proof connected components}. Finally, for $u\in U$, we have that
		\[
		\deg_{F_i}(u)=\begin{cases}
			\deg_{F_{i-1}}(u)&\text{ if }u\neq x_j,\\
			\deg_{F_{i-1}}(u)+1&\text{ if }u=x_j.
		\end{cases}
		\]
		In either case, using our inductive hypothesis and the fact that $\deg_{F_{i-1}}(x_j)<\deg_T(g(x_j))$, we have that $\deg_{F_i}(u)\leq \deg_T(g(u))$ for all $u\in U$, satisfying \ref{property tree proof degree of spine}.
		
		We can build this sequence of graphs until we reach a graph, say $F$, that satisfies~\ref{property tree proof H tree},~\ref{property tree proof connected components} and~\ref{property tree proof degree of spine}, and also satisfies~\eqref{equation tree proof degree inequality} with equality for all $u\in U$. Let $R:=V(F)\setminus U$. Then,~\ref{property tree proof H tree} and~\ref{property tree proof connected components} give that $(F,R)$ is a rooted $H$-tree with the connected components of $F[R]=F-U$ being proper induced subgraphs of $H$. By Observation~\ref{observation easy balanced condition} along with~\ref{property tree proof degree of spine}, we have further that $(F,R)$ is balanced and of rooted density $d$. Thus, by \Cref{theorem d large tree implies dadmissible} and the hypothesis that all proper induced subgraphs of $H$ are $d$-admissible, $H$ is $d$-admissible.
	\end{proof}
	
	A spiritually similar strategy to the proof of \Cref{lemma cut edge dadmissible}  can be made to work for arbitrary edges, though the argument must be more careful in order to  ensure that each component of $F[R]$ is a subgraph of $H$.
	
	\begin{lemma}\label{lemma arbitrary edge admissible}
		Let $k\in\mathbb{N}$. If $H$ is a connected graph which contains an edge $\hat{x}\hat{y}$ such that $\deg_{H}(\hat{x}),\deg_H(\hat{y})\le k$, and if $d\ge \max\{2k^2,\Delta(H)\}$ is a rational number such that every proper induced subgraph of $H$ is $d$-admissible, then $H$ is $d$-admissible.
	\end{lemma}
	
	Before getting into the details, let us briefly describe the high-level ideas of the proof.  Similar to \Cref{lemma cut edge dadmissible}, our $H$-tree $F$ will have a path of alternating $x_i,y_i$ vertices which corresponds to the path of some Bukh-Conlon tree $(T,R_T)$.  To this path we will add a large number of copies of $H$ that uses some edge of the path, giving some $H$-tree $F'$, say with $\c{H}$ the collection of copies of $H$ that we used in this construction. It remains to add copies of $H$ to this graph to increase the degrees of the $x_i,y_i$ vertices to match the degrees of $T$, and we can do this in two basic ways.
	
	The simplest way to increase the degree of some $x_i$ vertex is to add a new copy of $H$ which only intersects at $x_i$.  This move adds a new component to $F'[R]$ isomorphic to $H-\hat{x}$ (which is allowed), and increase the degree of $x_i$ by exactly $\deg_H(\hat{x})$, and hence is only allowed when $\deg_T(g(x_i))-\deg_{F'}(x_i)$ is large (i.e.\ when it is at least $\deg_H(\hat{x})$).  
	
	In order to handle the case where this degree difference is small, we will take some $H'\in \c{H}$ which contains $x_i$ and then add some new copy $H''$ of $H$ which agrees with $H'$ on all of $V(H)\sm \{\hat{y}\}$, increasing the degree of $x_i$ by exactly 1 (from the new copy of $\hat{y}$ in $H''$).  This moves removes a component of $F'[R']$ which was isomorphic to $H-\hat{x}-\hat{y}$ (namely that of $H'-x_i-y_{i+1}$ for some $i$) and replaces it with a component isomorphic to $H-\hat{x}$.  Note crucially that if we tried performing this same operation on the same $H'$, then this $H-\hat{x}$ component would gain an additional copy of $\hat{y}$, creating a component in $F'[R']$ which is not contained in $2^H\sm \{H\}$.  In order to avoid this situation, we will pre-assign each vertex $u'$ in the path a distinct collection $\c{H}_{u'}\sub \c{H}$ of $k-1$ copies for it to use for this second kind of move, which will suffice for us to perform this operation to distinct $H'\in \c{H}$ each time.  We now move onto the formal details of the proof.
	
	\begin{proof}
		Let $a,b$ be positive integers with $d=b/a$ such that $a$ is even and $a>k$. Let $(T,R_T)$ be a balanced tree of rooted density $d$ with $a$ unrooted vertices $u_1,u_2,\dots,u_a$, as given in \Cref{lem:BukhConlon}. Our goal is to construct an $H$-tree satisfying the conditions of \Cref{theorem d large tree implies dadmissible} to complete the proof.
		
		We begin by defining a set of $(a-1)k$ copies of $H$, indexed as $H_{i,j}$ with $i\in [a-1]$ and $j\in [k]$, as follows. First let us define each copy of the form $H_{i,1}$ for some $i\in [a-1]$. We start with $H_{1,1}$ an arbitrary copy of $H$. Iteratively, let us assume we have defined $H_{1,1},\ldots,H_{i-1,1}$. If $i$ is even with $2\le i\le a-1$, we define $H_{i,1}$ to be a copy of $H$ whose vertex playing the role of $\hat{y}$ is the same as that of $H_{i-1,1}$ and which is otherwise disjoint from all other vertices of every previous $H_{i',1}$.  Similarly, if $i$ is for odd and $2\le i\le a-1$ we define $H_{i,1}$ to be a copy whose vertex playing the role of $\hat{x}$ is the same as that of $H_{i-1,1}$ and which is otherwise disjoint from all other vertices of every previous $H_{i',1}$.   As in the previous proof, we define for each odd integer $1\le i\le a-1$ the vertex $x_i$ and $y_{i+1}$ to be the vertices of $H_{i,1}$ which plays the role of $\hat{x}$ and $\hat{y}$, respectively.  Let $U:=\{x_1,y_2,\ldots,x_{a-1},y_a\}$, and define $g:U \to V(T)\setminus R_T$ by $g(x_i)=u_i$ for odd $i$ and $g(y_i)=u_i$ for even $i$. 
		
		Having defined $H_{i,1}$ for all $1\le i\le a-1$, we then define $H_{i,j}$ for all $1\le i\le a-1$ and $2\le j\le k$ to be a copy of $H$ whose subgraph corresponding to $e=\hat{x}\hat{y}$ is equal to the subgraph of $H_{i,1}$ corresponding to $e$ and is otherwise pairwise disjoint from all other copies of $H$, i.e. we do this in such a way that the subgraphs $H_{i,j}-(U\cap V(H_{i,j})$ for all $i\in [a-1],j\in [k]$ are pairwise disjoint. We let 
		\[
		F':=\bigcup_{i\in [a-1],j\in [k]}^{}H_{i,j}
		\]
		denote the union of all of these copies. 
		
		It will be useful to have a linear ordering on the set of $H_{i,j}$'s. Towards this end, in a slight abuse of notation, for each $i\in [(a-1)k]$, let $H_i\subseteq \{H_{i',j'}\mid i'\in [a-1],j'\in [k]\}$ denote one of the copies of $H$ we added, in such a way that $H_i$ and $H_{i'}$ are distinct copies for each $i\neq i'$, and such that $H_i=H_{i,1}$ for each $i\in [a-1]$ (the remaining $H_{i,j}$'s can be ordered in any order).
		
		Similar to the proof of \Cref{lemma cut edge dadmissible}, we now want to add more copies of $H$ to $F'$ to increase the degrees of the vertices in $U$ to more closely match the degrees of the vertices in $g(U)$. We will do this in two rounds. In the first round, if there are any vertices $u\in U$ with 
		\begin{equation}\label{equation general edge degree bound less by k}
			\deg_{F'}(u)\leq \deg_{T}(g(u))-k,
		\end{equation}
		then we will add copies of $H$ to increase the degree of $u$ until~\eqref{equation general edge degree bound less by k} is not satisfied for any vertex $u\in U$. Formally, we will define a sequence $F_0\subseteq F_1\subseteq\dots$ with $F'=F_0$ together with a sequence $H_{(a-1)k+1},\ldots$ of copies of $H$ such that the following are satisfied for each $i$:
		
		\begin{enumerate}[label=(G\arabic*)]
			\item $F_i$ is an $H$-tree with defining copies $\{H_1,H_2,\dots,H_{(a-1)k+i}\}$,\label{property general graph proof H tree}
			\item The connected components of $F_i-U$ are all proper induced subgraphs of $H$,\label{property general graph proof connected components}
			\item The map $g$ is an isomorphism from $F_i'-U$ to $T-R_T$ and satisfies
			\begin{equation}\label{equation general graph proof degree inequality}
				\deg_{F_i}(u)\le \deg_T(g(u))
			\end{equation}
			for all $u\in U$.\label{property general graph proof degree of spine}
			\item For $j\in [(a-1)k]$, we have $\displaystyle V(H_j)\cap \bigcup_{j'\in [(a-1)k+i]\setminus \{j\}}H_{j'}\subseteq U$.\label{property general graph proof intersection of H_j}
		\end{enumerate}
		
		Before we describe how to form $F_i$ from $F_{i-1}$, let us prove that $F'=F_0$ satisfies~\ref{property general graph proof H tree} through~\ref{property general graph proof intersection of H_j}. That $F'$ is an $H$-tree follows from iterations of \Cref{lemma iteratively building H-trees}.  It is straightforward to see that the components of $F'-U$ are exactly the disjoint union of each $H_{i,j}-U$ which are all isomorphic to $H-\hat{x}-\hat{y}$. That $g$ is an isomorphism follows from $x_1y_2\cdots x_{a-1}y_a$ being a path.  For the degree condition, we see that for any $u\in U$,
		\[
		\deg_{F'}(u)\leq 2k(\max\{\deg_H(\hat{x}),\deg_H(\hat{y})\}-1)+2\le 2k^2\leq \deg_{T}(g(u)),
		\]
		with this last step using that each unrooted vertex in $T$ has degree at least $d\ge 2k^2$. Finally, the subgraphs $H_{i,j}-(U\cap V(H_{i,j})$ for all $i\in [a-1],j\in [k]$ are pairwise disjoint by construction. Thus, $F_0$ satisfies~\ref{property general graph proof H tree} through~\ref{property general graph proof intersection of H_j}.
		
		Now, via induction, for some $i\geq 1$, let us assume that $F_{i-1}$ satisfies~~\ref{property general graph proof H tree} through~\ref{property general graph proof intersection of H_j}. Furthermore, let us assume there exists some $u\in U$ such that $\deg_{F_{i-1}}(u)\leq \deg_T(g(u))-k$. We will further assume that $u=x_j$ for some odd $j$ (the case when $u=y_j$ is symmetric). Let $H_{(a-1)k+i}$ denote a copy of $H$ with $x_j$ playing the role of $\hat{x}$, but is otherwise disjoint from $F_{i-1}$, and let $F_i:=F_{i-1}\cup H_{(a-1)k+i}$. By \Cref{lemma iteratively building H-trees}, $F_i$ satisfies~\ref{property general graph proof H tree}. Furthermore, it is straightforward to see that $F_i-U\cong (F_{i-1}-U)\sqcup (H-\hat{x})$, so $F_i$ satisfies~\ref{property general graph proof connected components}. Furthermore, we have that for all $u\neq x_j$, $\deg_{F_i}(u)=\deg_{F_{i-1}}(u)\le\deg_{T}(g(u))$ by our inductive hypothesis, and for $x_j$, we have \[\deg_{F_i}(x_j)=\deg_{F_{i-1}}(x_j)+\deg_T(\hat{x})\leq \deg_{F_{i-1}}(x_j)+k\leq \deg_T(g(u)),\] so $F_i$ satisfies~\ref{property general graph proof degree of spine}. Finally, $H_{(a-1)k+i}$ intersects $F_{i-1}$ only in $x_j\in U$, so $F_i$ satisfies~\ref{property general graph proof intersection of H_j}.
		
		We can continue building our sequence of $F_i$'s until we reach some graph $F_{i^*}$ that satisfies~\ref{property general graph proof H tree} through~\ref{property general graph proof intersection of H_j}, and further has the property that $\deg_{F_{i^*}}(u)>\deg_T(g(u))-k$ for all $u\in U$. 
		
		At this point we begin the second round of adding copies of $H$. In this second round, we will add copies until~\eqref{equation general graph proof degree inequality} is satisfied with equality for all $u\in U$. Before we formally describe the process of adding more copies of $H$ to $F_{i^*}$, we need the following claim, which will allow us to guarantee that~\ref{property general graph proof connected components} stays satisfied.
		
		\begin{claim}\label{claim general edge proof we have disjoint collections}
			For all $1\le i\le a$, there exists pairwise disjoint collections $\c{H}_i\sub \{H_{i,j}\mid i\in [a-1],j\in [k]\}$ of size $k-1$ such that for each $i$, every copy in $\c{H}_i$ contains $x_i$ or $y_i$ if $i$ is odd or even, respectively.
		\end{claim}
		
		\begin{proof}
			Let $\c{H}_1$ consist of all the copies of the form $H_{1,j}$ with $1\le j\le k-1$.  For each $i$ with $2\le i \le k$ we let $\c{H}_i$ consist of the copies $H_{i-1,j}$ for $k-i+2\le j\le k$ together with the copies $H_{i,j}$ for $1\le j\le k-i$.  For all $k<i\le a$ we let $\c{H}_i$ consist of the copies $H_{i-1,j}$ for all $1\le j\le k-1$.  Note that in this definition we implicitly used that $a>k$ in order to guarantee that $H_{i,j}$ exists for all $2\le i\le k$.
			
			It is straightforward to check that each of the collections defined are pairwise disjoint (since e.g.\ each $\c{H}_{i-1}$ collection only uses $H_{i,j}$ with $1\le j \le k-i+1$) and have size exactly $k-1$ (since e.g.\ the $i$ with $2\le i\le k$ have size $(i-1)+(k-i)=k-1$). Furthermore, each $\mathcal{H}_i$ contains only copies of the form $H_{i,j}$ and $H_{i-1,j'}$, which intersect in $x_i$ or $y_i$, if $i$ is odd or even respectively.
		\end{proof}
		
		Now, starting from $F_{i^*}$, we continue to build a sequence of graphs $F_0\subseteq \dots\subseteq F_{i^*}\subseteq F_{i^*+1}\subseteq\dots$and copies $H_{(a-1)k+i}$; however for $i>i^*$, we no longer require $F_i$ to satisfy~\ref{property general graph proof intersection of H_j}. Instead, we still require $F_i$ to satisfy~\ref{property general graph proof H tree},~\ref{property general graph proof connected components},~\ref{property general graph proof degree of spine}, and also the following property:
		
		\begin{enumerate}[label=(G\arabic*), start=5]
			\item For each $j\in [a]$, with $u=x_j$ or $u=y_j$ depending on the parity of $j$, there are at least $\deg_T(g(u))-\deg_{F_i}(u)$ copies $H'\in \mathcal{H}_j$ of $H$ such that $V(H')\cap V(H'')\subseteq U$ for all $H''\in \{H_1,\dots,H_{(a-1)k+i}\}\setminus\{H'\}$.\label{property general graph proof keeping the disjoint choices}
		\end{enumerate}
		
		We first note that $\deg_T(g(u))-\deg_{F_{i^*}}(u)\leq k-1=|\mathcal{H}_j|$ for all $j\in [a]$ by the definition of $i^*$. This, along with the fact that $F_{i^*}$ satisfies~\ref{property general graph proof intersection of H_j} yields that $F_{i^*}$ also satisfies~\ref{property general graph proof keeping the disjoint choices}.
		
		Now, let us assume by induction that for some $i>i^*$, $F_{i-1}$ satisfies~\ref{property general graph proof H tree},~\ref{property general graph proof connected components},~\ref{property general graph proof degree of spine} and~\ref{property general graph proof keeping the disjoint choices}. Further, let us assume that there exists a vertex $u\in U$ such that $\deg_{F_{i-1}}(u)<\deg_{T}(g(u))$. Let $j\in [a]$ be such that $u=x_j$ or $u=y_j$, and let us assume that $u=x_j$ (the case where $u=y_j$ is symmetric). Let $H'\in \mathcal{H}$ be such that $V(H')\cap V(H'')\subseteq U$ for all $H''\in \{H_1,\dots,H_{(a-1)k+i}\}\setminus\{H'\}$ (which is guaranted to exist by~\ref{property general graph proof keeping the disjoint choices} and the fact that $\deg_{F_{i-1}}(u)<\deg_{T}(g(u))$), and define $H_{a(k-1)+i}$ to be a copy of $H$ such that $H_{a(k-1)+i}$ intersects $H'$ in the subgraph of $H'$ corresponding to $H-\hat{y}$, and is otherwise disjoint from $F_{i-1}$. Let $F_i=F_{i-1}\cup H_{(a-1)k+i}$.
		
		\Cref{lemma iteratively building H-trees} gives~\ref{property general graph proof H tree}. Furthermore, since $V(H')\cap V(H'')\subseteq U$ for all $H''\in \{H_1,\dots,H_{(a-1)k+i}\}\setminus\{H'\}$, we have that $H'$ is the only copy of $H$ in $\{H_1,\dots,H_{(a-1)k+i-1}\}$ that intersects $H_{(a-1)k+i}$ outside of $U$, giving us two consequences. The first consequence is that the components of $F_i-U$ are exactly the components of $F_{i-1}-U$, except in place of $H'-(U\cap V(H'))$, we have $H_{(a-1)k+i}-x_j\cong H-\hat{x}$, so~\ref{property general graph proof connected components} is satisfied. The second consequence being that ~\ref{property general graph proof keeping the disjoint choices} holds; for $j'\in [a]\setminus \{j\}$, the copies of $H$ in $\mathcal{H}_{j'}$ were not affected outside $U$, and only one copy in $\mathcal{H}_j$ was affected, while the degree of $x_j$ in $F_i$ is one higher than in $F_{i-1}$. Finally,~\ref{property general graph proof degree of spine} holds for $F_i$ since $H_{a(k-1)+i}$ does not affect the degree of any $u'\in U$, $u'\neq x_j$, and it increases the degree of $x_j$ by one (compared to $\deg_{F_{i-1}}(x_j)$), so since we had by assumption that $\deg_{F_{i-1}}(x_j)<\deg_T(g(x_j))$, we have that $\deg_{F_{i}}(x_j)\leq \deg_T(g(x_j))$.
		
		Thus, we can continue building this sequence in the second round until we reach a graph, call it $F$, that satisfies~\ref{property general graph proof H tree},~\ref{property general graph proof connected components} and~\ref{property general graph proof degree of spine}, and furthermore such that each vertex $u\in U$ satisfies~\eqref{equation general graph proof degree inequality} with equality. Let $R:=V(F)\setminus U$. We finish the proof by noting that $F$ satisfies the conditions for \Cref{theorem d large tree implies dadmissible}.  Indeed, \ref{property general graph proof H tree} and~\ref{property general graph proof connected components} give that $(F,R)$ is a rooted $H$-tree with the connected components of $F[R]=F-U$ being proper induced subgraphs of $H$. By Observation~\ref{observation easy balanced condition} along with~\ref{property general graph proof degree of spine}, we have further that $(F,R)$ is balanced and of rooted density $d$.  Finally by hypothesis all proper induced subgraphs of $H$ are $d$-admissible, so $H$ is $d$-admissible by \Cref{theorem d large tree implies dadmissible}.
	\end{proof}
	
	We will now show how to use \Cref{lemma cut edge dadmissible} and \Cref{lemma arbitrary edge admissible} to prove Theorems~\ref{theorem:maxDegreeGeneral} and \ref{theorem:maxDegreeTrees}.  In fact, we can prove a common generalization of these two theorems via the following technical definition.
	\begin{definition}
		We say that a graph $H$ is \textbf{$k$-edge-degenerate} for some non-negative integer $k$ if every non-empty $H'\sub H$ either contains a cutedge or contains an edge $e=\hat{x}\hat{y}$ with $\deg_{H'}(\hat{x}),\deg_{H'}(\hat{y})\le k$.
	\end{definition}
	
	This definition can be thought of as a technical variant of the usual notion a graph being $k$-degenerate which is defined as saying that every induced subgraph contains a vertex of degree at most $k$.  We note for our forthcoming proof that if $H$ is $k$-edge degenerate then every subgraph of $H$ is also $k$-edge degenerate.
	
	\begin{theorem}\label{thm:degenerate}
		If $H$ is a graph of maximum degree at most $\Delta$ which is $k$-edge-degenerate, then every rational in
		\[
		\left[v(H)-\frac{e(H)}{\max\{2k^2,\Delta\}},\ v(H)\right]
		\]
		is realizable for $H$.
	\end{theorem}
	
	We note that this result immediately implies Theorems~\ref{theorem:maxDegreeGeneral} and \ref{theorem:maxDegreeTrees}, as every graph $H$ of maximum degree $\Delta$ is trivially $\Delta$-edge-degenerate, and every tree is trivially $0$-edge-degenerate since every edge of a tree is a cutedge.  More generally, \Cref{thm:degenerate} can be used to prove that a number of other graphs besides just trees have every rational above $v(H)-e(H)/\Delta$ being realizable.  For example, this holds for the graph $H$ consisting of $t\ge 4$ triangles all sharing a common vertex, as this graph can easily be seen to have maximum degree $\Delta=2t$ and to be $2$-edge-degenerate.
	
	\begin{proof}
		We will prove the stronger fact that every such $H$ is $d$-admissible for every rational $d\ge \max\{2k^2,\Delta\}$, which implies this statement by \Cref{observation admissible implies realizable}.  Indeed, assume there exist some parameters $k,d,\Delta$ such that this fails for some $H$, and choose such an $H$ with as few vertices as possible.  In particular, because each induced subgraph of $H$ is $k$-edge-degenerate and has maximum degree at most $\Del$, each proper induced subgraph of $H$ must be $d$-admissible by the minimality of $H$.
		
		Note that if $H$ is disconnected then each connected component is $d$-admissible by our comments above, so \Cref{observation disjoint union of dadmissible} implies that $H$ is $d$-admissible, a contradiction.  We may thus assume that $H$ is connected.  Similarly, if $H$ is a single vertex then it is $d$-admissible for all $d$, so we may assume that $H$ contains at least two vertices.
		
		Since $H$ is $k$-edge-degenerate, there exists some edge $\hat{x}\hat{y}$ of $H$ satisfying either (a) $\deg_H(\hat{x})\leq k$ and $\deg_H(\hat{y})\leq k$ or (b) the edge is a cutedge. If (a) occurs then $H$ is $d$-admissible by \Cref{lemma arbitrary edge admissible}, while if (b) occurs, then \Cref{lemma cut edge dadmissible} implies that $H$ is $d$-admissible, giving the desired contradiction.
	\end{proof}
	
	\section{Non-Realizable Exponents}\label{section non-realizable}
	
	This section is organized as follows.  In \Cref{subsection non-realizable tools} we state without proof the main tools that we need for showing the non-realizability of rational exponents.  We prove almost all of our results in \Cref{subsection non-realizable main} except for one technical result \Cref{theorem:treeReduction} which we postpone until \Cref{subsection tree reduction}.
	
	\subsection{Key Tools}\label{subsection non-realizable tools}
	
	Our approach for non-realizability of exponents for graphs $H$ is inspired by the proof of non-realizability for stars given in~\cite{EHK2025}. In their proof, they make the simple but crucial observation that if $\c{F}$ is a family which does not contain a subgraph of a star, then $\ex(n,K_{1,t-1},\c{F})=\Om(n^{t-1})$ since the graph $G=K_{1,n-1}$ is $\c{F}$-free.  On the other hand, if $\c{F}$ does contain a star, then $\ex(n,K_{1,t-1},\c{F})=O(n)$ since each vertex will be in at most $O(1)$ stars.  From this one deduces that no exponent in $(1,t-1)$ is realizable for $K_{1,t-1}$.
	
	Similar to the argument above, we will prove the non-realizability of exponents for graphs $H$ by constructing a ``simple'' family of graphs $\c{F}_H$ such that the behavior of $\ex(n,H,\c{F})$ will largely depend on whether or not $\c{F}$ contains (a subgraph of) each graph in $\c{F}_H$.   To this end, we make the following crucial definition.
	
	\begin{definition}\label{definition flower powers}
		Given a graph $H$, a set of vertices $R\sub V(H)$, and an integer $\pow$, we define the graph $H_R^\pow$ to be the graph obtained by taking the union of $\pow$ copies of $H$ such that each of the copies agree on the set of vertices $R$ and are otherwise disjoint.  For each integer $k\ge 1$ we define
		\[
		\c{F}_{H,k}^\pow:=\{H_R^\pow: R\subsetneq V(H),\ H-R\textrm{ has at least }k\textrm{ connected components}\}.
		\]
	\end{definition}
	
	See Figure~\ref{figure flower power definition full figure} for an example of a graph of the form $H_R^\pow$.
	
	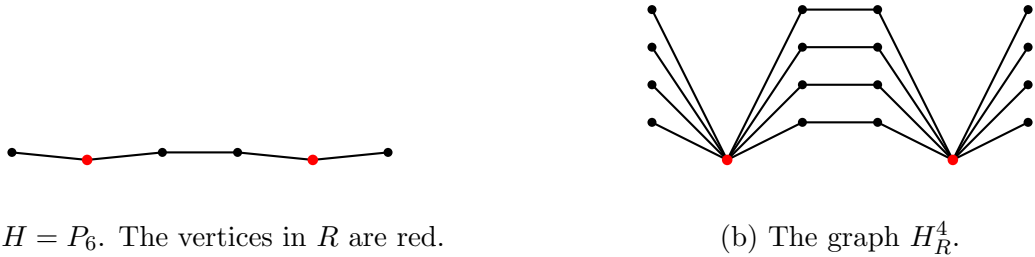
\begin{figure}[hb!]
		\begin{center}
			\begin{subfigure}[t]{0.4\textwidth}
				\begin{center}\begin{tikzpicture}
						\draw[thick] (0,0.1)--(1,0)--(2,0.1)--(3,0.1)--(4,0)--(5,0.1);
						
						\node (0) at (0,0.1) {};
						\node (1) at (1,0) {};
						\node (2) at (2,0.1) {};
						\node (3) at (3,0.1) {};
						\node (4) at (4,0) {};
						\node (5) at (5,0.1) {};
						
						\fill (0) circle (0.06) node [above] {};
						\fill[color=red] (1) circle (0.07) node [above] {};
						\fill (2) circle (0.06) node [above] {};
						\fill (3) circle (0.06) node [above] {};
						\fill[color=red] (4) circle (0.07) node [above] {};
						\fill (5) circle (0.06) node [above] {};
					\end{tikzpicture}
				\end{center}
				\caption{$H=P_6$. The vertices in $R$ are red.}\label{figure flower power definition path}
			\end{subfigure}\hspace{0.1\textwidth}\begin{subfigure}[t]{0.4\textwidth}
				\begin{center}
					\begin{tikzpicture}
						\draw[thick] (0,0.5)--(1,0)--(2,0.5)--(3,0.5)--(4,0)--(5,0.5);
						\draw[thick] (0,1)--(1,0)--(2,1)--(3,1)--(4,0)--(5,1);
						\draw[thick] (0,1.5)--(1,0)--(2,1.5)--(3,1.5)--(4,0)--(5,1.5);
						\draw[thick] (0,2)--(1,0)--(2,2)--(3,2)--(4,0)--(5,2);
						
						\node (0) at (0,0.5) {};
						\node (1) at (1,0) {};
						\node (2) at (2,0.5) {};
						\node (3) at (3,0.5) {};
						\node (4) at (4,0) {};
						\node (5) at (5,0.5) {};
						
						\node (6) at (0,1.5) {};
						\node (7) at (2,1.5) {};
						\node (8) at (3,1.5) {};
						\node (9) at (5,1.5) {};
						
						\node (10) at (0,1) {};
						\node (11) at (2,1) {};
						\node (12) at (3,1) {};
						\node (13) at (5,1) {};
						
						\node (14) at (0,2) {};
						\node (15) at (2,2) {};
						\node (16) at (3,2) {};
						\node (17) at (5,2) {};
						
						\fill (0) circle (0.06) node [above] {};
						\fill[color=red] (1) circle (0.07) node [above] {};
						\fill (2) circle (0.06) node [above] {};
						\fill (3) circle (0.06) node [above] {};
						\fill[color=red] (4) circle (0.07) node [above] {};
						\fill (5) circle (0.06) node [above] {};
						\fill (6) circle (0.06) node [above] {};
						\fill (7) circle (0.06) node [above] {};
						\fill (8) circle (0.06) node [above] {};
						\fill (9) circle (0.06) node [above] {};
						\fill (10) circle (0.06) node [above] {};
						\fill (11) circle (0.06) node [above] {};
						\fill (12) circle (0.06) node [above] {};
						\fill (13) circle (0.06) node [above] {};
						\fill (14) circle (0.06) node [above] {};
						\fill (15) circle (0.06) node [above] {};
						\fill (16) circle (0.06) node [above] {};
						\fill (17) circle (0.06) node [above] {};
					\end{tikzpicture}
				\end{center}
				\caption{The graph $H_R^4$.}\label{figure flower power definition H_R^4}
			\end{subfigure}
		\end{center}
		\caption{An example of the graph $H_R^4$ from \Cref{definition flower powers} when $H$ is the path on $6$ vertices and $R$ consists of the set of red vertices in Figure~\ref{figure flower power definition path}. $H_R^4$ is a member of the family $\mathcal{F}_{H,3}^4$ since $H-R$ has three connected components.}\label{figure flower power definition full figure}
	\end{figure}

	Note that $H_R^\pow$ can also be viewed as the element of $(H,R)^\pow$ which has the maximum number of vertices.  The key fact we need regarding non-realizable exponents is the following.
	\begin{proposition}\label{proposition:keyObservation}
		For every graph $H$, integer $k\ge 1$, and family of graphs $\c{F}$, either $\ex(n,H,\c{F})=\Om_{H,k}(n^k)$ or there exists an integer $q$ such that every $\c{F}$-free graph is also $\c{F}_{H,k}^q$-free.  In particular, either $\ex(n,H,\c{F})=\Om_{H,k}(n^k)$ or $\ex(n,H,\c{F}) \leq \ex(n,H,\c{F}_{H,k}^\pow)$.
	\end{proposition}
	
	This result immediately gives the following, which will be the foundation of every non-realizability result in this paper.
	\begin{corollary}\label{corollary:keyObservation}
		If $H$ is a graph, $k\ge 1$ is an integer, and $\re>0$ is real such that for all $\pow$ we have $\ex(n,H,\c{F}_{H,k}^\pow)=o_\pow(n^{\re})$, then no rational in $[\re,k)$ is realizable for $H$.
	\end{corollary}
	
	We have now reduced the problem of showing non-realizability of exponents to proving upper bounds on $\ex(n,H,\c{F}_{H,k}^\pow)$.   For example, if $H=K_{1,t-1}$ then one can check that $\c{F}_{H,t-1}^\pow$ consists only of a star $K_{1,q'}$, so by \Cref{corollary:keyObservation} one can show that no exponent in $(1,t-1)$ is realizable simply by showing that $\ex(n,K_{1,t-1},K_{1,\pow'})=O_{\pow'}(n)$, which is easy to do.
	
	Estimating $\ex(n,H,\c{F}_{H,k}^\pow)$ seems to be a difficult problem in general.  However, in the case of trees we can significantly reduce the difficulty through the following.
	
	\begin{theorem}\label{theorem:treeReduction}
		If $T\ne K_1$ is a tree and if $\c{F}$ is a family of graphs such that every $\c{F}$-free graph is also $\c{F}_{T,k}^q$-free for some integers $k,\pow\ge 1$, then
		\[
		\ex(n,T,\c{F})=O_{T,k,q}(\ex(n,\c{F})^{k-1}).
		\]
	\end{theorem}
	That is, for trees we can reduce upper bounding the \textit{generalized} Tur\'an number $\ex(n,T,\c{F})$ to upper bounding the \textit{classical} Tur\'an number $\ex(n,\c{F})$, which is often significantly easier to work with.
	
	There are some cases where the bounds of \Cref{theorem:treeReduction} can be sharp, such as when\footnote{In this case avoiding $\c{F}_{H,k}^\pow$ is essentially equivalent to avoiding $K_{2,t}$, so \Cref{theorem:treeReduction} gives $\ex(n,P_5,\c{F}_{T,k}^\pow)=O((n^{3/2})^2)=O(n^3)$, which is best possible by considering any nearly-regular $C_4$-free graph on $\Theta(n^{3/2})$ edges.} $H=P_5$ and $k=3$.  However, it is worth noting that in general \Cref{theorem:treeReduction} can be far from sharp since, in particular, it can never give a bound stronger than that of $O(n^{k-1})$ even when the true behavior of $\ex(n,H,\c{F})$ may be as small as $O(n)$; such as when $H=K_{1,t-1}$, $k=t-1$, and $\c{F}=\c{F}_{H,k}^\pow$, for example.
	
	\subsection{Proof of Main Results}\label{subsection non-realizable main}
	In this subsection we prove all of our main results except for \Cref{theorem:treeReduction} which we defer until \Cref{subsection tree reduction} due to its rather technical proof.  We begin by proving our key observation \Cref{proposition:keyObservation}, which to a large extent boils down to showing the following basic fact.
	
	\begin{lemma}\label{lemma powers many copies}
		If $H$ is a graph and $R\subsetneq V(H)$ is such that $H-R$ contains $k$ connected components, then the graph $H_R^n$ contains at least $n^k$ copies of $H$.
	\end{lemma}
	\begin{proof}
		Recall that $H_R^n$ is defined to be the union of $n$ copies $H_1,\ldots,H_n$ of $H$ which all agree on $R$ and which otherwise are disjoint from each other.  Consider the following way of embedding a copy of $H$ into $H_R^n$.  We first embed $R\sub V(H)$ into the common intersection of the $H_i$.  Then, to embed each connected component of $H-R$, we choose some integer $i$ and embed this component into the corresponding connected component of $H_i - R$.  
		
		It is not difficult to see that this always defines an embedding of $H$ with distinct embeddings corresponding to distinct copies of $H$.  Moreover, for this embedding we have a total of $n$ choices for each of the $k$ connected components of $H-R$, showing that the number of copies of $H$ is at least $n^{k}$ as desired.
	\end{proof}

	\begin{proof}[Proof of \Cref{proposition:keyObservation}]
		Let $H$, $k \geq 1$, and $\c{F}$ be given. For $R \subseteq V(H)$, let $w(R)$ denote the number of connected components of $H - R$. 
		We consider two cases.
		
		\textbf{Case 1:} for every $R \subsetneq V(H)$ with $w(R) \geq k$, there exists some integer $\pow(R)$ such that the family $\c{F}$ contains a graph which is a subgraph of $H_R^{\pow(R)}$.  Letting $\pow:=\max_R \pow(R)$, this means that any $\c{F}$-free graph is also $\c{F}_{H,k}^\pow$-free as desired.
		
		\textbf{Case 2:} there exists $R \subsetneq V(H)$ with $w(R) \geq k$ such that $\c{F}$ does not contain any subgraph of $H_R^\pow$ for any $\pow$. Let $n' = \lfloor{n/|V(H)|}\rfloor$ and consider $G = H_R^{n'}$.  This graph has at most $n$ vertices, is $\c{F}$-free by hypothesis, and it contains at least $\Om_{H,k}(n^k)$ copies of $H$ by  \Cref{lemma powers many copies}, proving $\ex(n,H,\c{F}) = \Om(n^k)$. 
		
	\end{proof}
	
	We next prove \Cref{proposition no 01 exponents}, which we recall says that every rational in $(0,1)$ is non-realizable for every $H$ with $0,1$ being realizable under certain hypothesis.  For this, we establish two basic lemmas that will further help us in proving \Cref{theorem:treeReduction}.
	
	\begin{lemma}\label{lemma:canonicalReduction}
		If $G$ and $H$ are graphs and if $\Phi\subseteq \mathrm{Mon}(H,G)$ is a collection of monomorphisms, then there exists a partition of $V(G)$ into sets $\{V_{\hat{x}}:\hat{x}\in V(H)\}$ and a subcollection $\Phi'\sub \Phi$ with $|\Phi'|\ge v(H)^{-v(H)}|\Phi|$ such that $\phi(\hat{x})\in V_{\hat{x}}$ for all $\hat{x}\in V(H)$ and $\phi\in \Phi'$.
	\end{lemma}
	
	\begin{proof}
		Construct the sets $V_{\hat{x}}$ randomly by assigning each $u\in V(G)$ to each possible set uniformly and independently.  Letting $\Phi'\sub \Phi$ be the (random) collection of $\phi\in \Phi$ satisfying $\phi(\hat{x})\in V_{\hat{x}}$ for all $\hat{x}\in V(H)$, we see that $\Pr[\phi\in \Phi']=v(H)^{-v(H)}$ and hence $\E[|\Phi'|]=v(H)^{-v(H)}|\Phi|$ by linearity of expectation.  In particular, there exists some choice of $V_{\hat{x}}$ sets such that $|\Phi'|\ge v(H)^{-v(H)}|\Phi|$, giving the result.
	\end{proof}

	\begin{lemma}\label{lemma:sunflower}
		For every graph $H$ and integer $\pow \ge 2$, there exists some integer $\Pow=\Pow(H,q)$ such that $\Pow\ge \pow$ and such that if $G$ is a graph and if $\Phi\subseteq \mathrm{Mon}(H,G)$ satisfies $|\Phi|\ge \Pow$, then there exists a set $R\subsetneq V(H)$ and distinct maps $\phi_1,\ldots,\phi_\pow\in \Phi$ such that for all $i\ne j$, we have $\phi_i(\hat{x})=\phi_j(\hat{y})$ if and only if $\hat{x}=\hat{y}$ and $\hat{x}\in R$.
	\end{lemma}
	In other words, the maps $\phi_1,\ldots,\phi_q$ are such that their image is a copy of $H_R^\pow$ in $G$.
	\begin{proof}
		This result is a consequence of the Erd\H{o}s-Rado sunflower lemma~\cite{ER1960}, which states\footnote{We note in passing that it is a major open problem to improve upon the bounds of the Erd\H{o}s-Rado sunflower lemma, though we will only need that some finite bound exists.  For a more thorough treatment of this result and best known bounds we refer the reader to~\cite{ALWZ2021}.} that if $\c{H}$ is an $r$-uniform hypergraph with $|\c{H}|\ge k!(r-1)^k$, then there exist edges $e_1,\ldots,e_k\in \c{H}$ and a set $K\sub V(\c{H})$ such that $e_i\cap e_j=K$ for every $i\ne j$. We will ultimately apply this result with $r:=v(H)$ and $k:=q\cdot v(H)!$, with us in the end proving the lemma for the value $\Pow(H,G):=v(H)!\cdot k!(r-1)^k$.
		
		Given a collection of at least $\Pow(H,G)$ monomorphisms $\Phi\subseteq \mathrm{Mon}(H,G)$, we first take a subcollection $\Phi'\subseteq \Phi$ such that $\phi(V(H))\neq \psi(V(H))$ for all $\phi,\psi\in \Phi'$. Since at most $v(H)!$ monomorphisms can have the same image, we can find such a collection with $|\Phi'|\geq |\Phi|/v(H)!\geq k!(r-1)^k$.
		
		Consider the $r$-uniform hypergraph $\mathcal{H}$ with $V(\mathcal{H})=V(G)$ and
		\[
		E(\mathcal{H})=\{\phi(V(H))\mid \phi\in \Phi'\}.
		\]
		Note that $|\mathcal{H}|=|\Phi'|$ by the defining property of $\Phi'$. By the sunflower lemma, there are $k$ edges of $\mathcal{H}$, say $\phi_1(V(H)),\dots,\phi_k(V(H))$, and a set $K\subseteq V(G)$ such that $\phi_i(V(H))\cap \phi_j(V(H))=K$ for all $i\neq j$. Let $K=\{v_1,v_2,\dots,v_t\}$, and associate to each $\phi_i$ the tuple $(\phi_i^{-1}(v_1),\phi_i^{-1}(v_2),\dots,\phi_i^{-1}(v_t))$. Since the $\phi_i$'s are injective and $t\leq v(H)$, we have that there are at most $v(H)!$ possible tuples, so by pigeonhole there must exist a collection of at least $k/v(H)!=q$ of the $\phi_i$'s which correspond to the same tuple.  Without loss of generality we can assume that $\phi_1,\dots,\phi_q$ have this property. Let $R:=\phi_1^{-1}(K)$, noting that by definition we have for $1\leq i,j\leq q$, $i\neq j$  that $\phi_i(\hat{x})=\phi_j(\hat{y})$ if and only if $\hat{x}=\hat{y}$ and $\hat{x}\in R$.  This implies that $R$ and the $\phi_1,\ldots,\phi_q$ satisfy all the properties of the lemma except possibly the property that $R\subsetneq V(H)$, but this follows from the fact that the $\phi_i$ maps are distinct and $q\ge 2$.

	\end{proof}
	
	Putting these two results together gives the following.
	
	\begin{lemma}\label{proposition FH1 Bounded} 
		For any graph $H$, we have $\ex(n,H,\c{F}_{H,1}^q)=O_q(1)$.
	\end{lemma}
	\begin{proof}
		Letting $\Pow=O_{\pow}(1)$ be as in \Cref{lemma:sunflower}, we see that if $G$ is a graph such that $\mathrm{mon}(H,G)\geq \Pow$, then $G$ contains a copy of $H_R^\pow$ for some $R\subsetneq V(H)$. Since $H-R$ has at least 1 connected component, we have $H_R^\pow\in \c{F}_{H,1}^\pow$. Thus, by Observation~\ref{observation monomorphisms same as copies}, no $\c{F}_{H,1}^\pow$-free graph can have more than $\Pow/\mathrm{aut}(H)$ copies of $H$, implying that $\ex(n,H,\c{F}_{H,k}^\pow)=O_q(1)$ as desired.
	\end{proof}
	
	We now give our proof of \Cref{proposition no 01 exponents}.
	\begin{proof}[Proof of \Cref{proposition no 01 exponents}]
		The fact that no rational in $(0,1)$ is realizable for any graph $H$ follows from \Cref{proposition FH1 Bounded} and \Cref{corollary:keyObservation}.  It remains to show that $0$ and $1$ are realizable for appropriate graphs $H$.  For both of these results, let $\c{F}_H$ be the (finite) family consisting of every graph which can be obtained by adding a new vertex $v$ to $H$ such that $v$ is made adjacent to exactly one vertex of $H$.  
		
		Observe that in any $\c{F}_H$-free graph, no connected component containing a copy of $H$ can have more than $v(H)$ vertices.  If $H$ is connected then this implies $\ex(n,H,\c{F}_H)=O(n)$ since each copy of $H$ in an $\c{F}_H$-free graph must lie within one of the at most $n$ connected components of size $O(1)$.  On the other hand, we have $\ex(n,H,\c{F}_H)=\Om(n)$ by considering the disjoint union of cliques of size $v(H)$, proving that 1 is realizable in this case.
		
		To show that $0$ is realizable for any graph $H$ without isolated vertices, we consider the family $\c{F}=\c{F}_H\cup \{H\cup K_2\}$.  By construction, any $\c{F}$-free graph $G$ containing a copy of $H$ on some set of vertices $S$ must have all of its edges lie between two vertices of $S$, implying that every copy of $H$ must be contained in this set $S$ since $H$ has no isolated vertices.  Since $|S|=v(H)$ this implies $\ex(n,H,\c{F})=O(1)$.  On the other hand, for $n\ge v(H)$ the graph $G$ consisting of one copy of $H$ and $n-v(H)$ isolated vertices is $\c{F}$-free, showing $\ex(n,H,\c{F})=\Om(1)$.
	\end{proof}

	It remains to prove our results for trees assuming the validity of \Cref{theorem:treeReduction}.  We begin by deriving our weak version of this result.
	\begin{proof}[Proof of \Cref{theorem weak tree reduction}]
		Recall that we aim to prove that if $T\ne K_1$ is a tree and if $\c{F}$ is a family of graphs, then for every integer $k\ge 1$ either $\ex(n,T,\c{F})=\Om_{T,k}(n^k)$ or $\ex(n,T,\c{F})=O_{T,k,\c{F}}(\ex(n,\c{F})^{k-1})$.  And indeed, by \Cref{proposition:keyObservation}, if $\ex(n,T,\c{F})=\Om_{T,k}(n^k)$ does not hold then there exists some $q\ge 1$ such that every $\c{F}$-free family is also $\c{F}_{T,k}^q$-free.  It follows from \Cref{theorem:treeReduction} that  $\ex(n,T,\c{F})=O_{T,k,q}(\ex(n,\c{F})^{k-1})$, proving the result.
	\end{proof}
	
	For our next result, we make the following basic observation.
	\begin{lemma}\label{lemma alpha empty}
		If $H$ is a graph and $k>\al(H)$, then $\c{F}_{H,k}^q=\emptyset$ for all $q$.
	\end{lemma}
	\begin{proof}
		Assume for contradiction that this set is non-empty.  This implies that there exists a set $R\sub V(H)$ such that $H-R$ has at least $k$ connected components.  Taking one vertex from each of these components defines an independent set in $H$ of size at least $k>\al(H)$, a contradiction to the definition of $\al(H)$.
	\end{proof}
	
	We next derive our first corollary of \Cref{theorem weak tree reduction}. 
	\begin{proof}[Proof of \Cref{corollary trees vs forests}]
		Recall that we aim to prove that if $T\ne K_1$ is a tree and if $\c{F}$ is a family of graphs which contains a forest, then either $\ex(n,T,\c{F})=\Theta(n^k)$ for some integer $k\le \al(T)$ or $\ex(n,T,\c{F})=0$ for all sufficiently large $n$.
		
		First, observe that if $\c{F}$ contains any graph which is a subgraph of $T$ together with some number of isolated vertices, then $\ex(n,T,\c{F})=0$ for all $n$ sufficiently large.  We can thus assume that this is not the case, and hence $\ex(n,T,\c{F})=\Om(1)$ by considering the $n$-vertex graph $G$ which consists of a single copy of $T$ together with $n-v(T)$ isolated vertices.
		
		Let $k\ge 0$ denote the smallest integer such that $\ex(n,T,\c{F})=O(n^k)$.  We claim that $\ex(n,T,\c{F})=\Theta(n^k)$.  Indeed, if $\ex(n,T,\c{F})=\Om(n^k)$ then we are done, so we can assume this is not the case.  This together with our assumption $\ex(n,T,\c{F})=\Om(1)$ and \Cref{proposition:keyObservation} implies both that $k\ge 1$ and that every $\c{F}$-free graph is also $\c{F}_{T,k}^q$ for some $q\ge 1$. Thus by  \Cref{theorem:treeReduction} we find that
		\[\ex(n,T,\c{F})=O(\ex(n,\c{F})^{k-1})=O(n^{k-1}),\]
		where this last equality used our hypothesis that $\c{F}$ contains a forest together with the standard fact that Tur\'an numbers of forests are at most $O(n)$.  This inequality contradicts the minimality of our choice of $k$, giving the claim.
		
		It only remains now to show that $k\le \alpha(T)$.  For this, because $\c{F}_{H,\alpha(H)+1}^q=\emptyset$ by \Cref{lemma alpha empty}, we have that every $\c{F}$-free graph is trivially $\c{F}_{H,\alpha(H)+1}^q$-free for any $q$.  It then follows from \Cref{theorem:treeReduction} that $\ex(n,T,\c{F})=O(\ex(n,\c{F})^{\alpha(T)})=O(n^{\alpha(T)})$, completing the proof.
		
	\end{proof}
	
	Lastly, we prove our stability-type result \Cref{theorem:leafGap}.

	\begin{proof}[Proof of \Cref{theorem:leafGap}]
		Recall that we aim to prove that if $T\ne K_2$ is a tree with $\ell\ge 2$ leaves and $\ex(n,T,\c{F})=O(n^{\ell})$, then $\ex(n,T,\c{F})=\Theta(n^k)$ for some integer $k$ or is identically 0 for large $n$.  If $\c{F}$ contains a forest then the result follows from \Cref{corollary trees vs forests}, so we can assume every element of $\c{F}$ contains a cycle.  
		
		Let $n':=\lfloor n/|V(T)|\rfloor$ and consider the graph $G=T_R^{n'}$ with $R$ the set of non-leaves of $T$ (i.e.\ $G$ is the graph obtained by duplicating each leaf of $T$ a total of $n'$ times).  It is not difficult to see that $G$ is a tree and hence is $\c{F}$-free, and by \Cref{lemma powers many copies} it contains at least $\Omega(n^\ell)$ copies of $T$ since every leaf of $T$ is a connected component of $T-R$ (here we implicitly use that $T\ne K_2$ to ensure that the unique neighbor of each leaf is not a leaf and hence in $R$).  This combined with the assumption $\ex(n,T,\c{F})=O(n^{\ell})$ implies $\ex(n,T,\c{F})=\Theta(n^\ell)$, proving the result.
	\end{proof}
	
	\subsection{Proof of \texorpdfstring{\Cref{theorem:treeReduction}}{Theorem 5.3}}\label{subsection tree reduction}
	Our main goal for this section will be to prove the following technical result.
	
	\begin{theorem}\label{theorem:treeBoundEG}
		If $T\ne K_1$ is a tree and if $k,\pow\ge 2$ are integers, then any $n$-vertex graph $G$ which is $\c{F}_{T,k}^\pow$-free contains at most $O(e(G)^{k-1})$ copies of $T$. 
	\end{theorem}
	
	Our motivation for this is that it quickly implies \Cref{theorem:treeReduction}.
	\begin{proof}[Proof of \Cref{theorem:treeReduction}]
		The case $k=1$ follows from \Cref{proposition FH1 Bounded}. For $k,\pow\ge 2$; by taking $G$ to be an $\c{F}$-free graph with $\ex(n,T,\c{F})$ copies of $T$, we find by \Cref{theorem:treeBoundEG} that
		\[\ex(n,T,\c{F})=O(e(G)^{k-1})=O(\ex(n,\c{F})^{k-1}),\]
		with this last step holding because $G$ (being an $\c{F}$-free graph) is $\c{F}_{T,k}^q$-free by hypothesis. The only remaining case then is $\pow=1$, which follows from the case $\pow=2$ and monotonicity.
	\end{proof}

	\textbf{Proof Intuition for \Cref{theorem:treeBoundEG}}.  The statement of \Cref{theorem:treeBoundEG} suggests how it might be proven; namely by taking each copy of $T$ in a $\c{F}_{T,k}^\pow$-free graph $G$ and having each copy ``correspond'' to a set of $k-1$ edges $E$ of $G$ in such a way that each set of edges is corresponded to by at most $O(1)$ copies of $T$.  
	
	For example, consider $P_5$ the path graph on $x_1x_2x_3x_4x_5$ and $G:=(P_5)_{\{x_3,x_4\}}^n$ (i.e.\ $G$ is obtained by taking $n$ copies of $P_5$ which all agree on the edge $x_3x_4$ and are otherwise disjoint).   In this example, each copy $P$ of $P_5$ in $G$ can be uniquely identified by specifying the edges of $G$ which play the role of $x_1x_2$ and $x_4x_5$ in $P$, implying that $G$ contains at most $e(G)^2$ copies of $P_5$.  The same argument works if we alternatively specified which edges play the role of $x_2x_3$ and $x_4x_5$ in $P$.
	
	In the proof above it is  critical that we specify the edge of $G$ playing the role of $x_4x_5$, as even if we specify every other edge of a copy of $P_5$, there would still exist $n$ ways to extend these specified edges into distinct copies of $P_5$.  More generally, the edge set $E$ that we wish to choose for each copy must be capable of ``detecting'' every way in which its copy has ``many options'' for being extended.
	
	Motivated by the above, for each monomorphism $\phi:V(T)\to V(G)$ of a tree $T$, we wish to define a set of subtrees $\c{T}_\phi$ of $T$ which have ``many extensions'' with respect to $\phi$, in the sense that for each $T'\in \c{T}_\phi$, there exist many $\phi'$ which agree with $\phi$ outside of $T'$ (i.e.\ there exist many ways of changing how $\phi$ maps $T'$ while maintaining that $\phi$ is a monomorphism).  For example, if $G=(P_5)^n_{\{x_3,x_4\}}$ then for every choice of $\phi$ we will define $\c{T}_\phi$ to consist of the subtrees which either contain $x_1x_2x_3$ or $x_4x_5$, i.e.
	\[
	\c{T}_{\phi}=\{x_1x_2x_3,x_1x_2x_3x_4,x_1x_2x_3x_4x_5,x_2x_3x_4x_5,x_3x_4x_5,x_4x_5\}.
	\]
	More generally, if $G=T_R^n$, then $\c{T}_\phi$ will always consist of the subtrees that contain some connected component of $T-R$ together with the vertices of $R$ adjacent to this component.
	
	The key insight about working with these sets $\c{T}_\phi$ is that whatever edge set $E$ we ultimately choose to identify each copy of $\phi$ with, this set $E$ must have the property that every $T'\in \c{T}_\phi$ contains an edge which ``corresponds'' to an edge in $E$.  Indeed, if this were not the case then $E$ would necessarily be contained in ``many'' copies of $T$, namely those which agree with $\phi$ outside of the subtree $T'\in \c{T}_\phi$ which does not contain an edge corresponding to an edge in $E$.
	
	Given all this, we are led to the following sub-question: what conditions can one impose on a set of subtrees $\c{T}$ to guarantee the existence of a set of $k-1$ edges $E$ that contains an edge from each element of $\c{T}$?  In particular, when $k=2$, this question asks when we can guarantee that a collection of subtrees $\c{T}$ contains an edge which lies in every subtree of $\c{T}$.  Such a statement is reminiscent of the classical Helly property for trees, which states that if $\c{T}$ is a collection of subtrees of a given tree $T$ whose \textit{vertex} sets pairwise intersect, then there exists some \textit{vertex} which is contained in every element of $\c{T}$.   For our setting, we will need to prove an analogous Helly-type statement with respect to the \textit{edge} sets of subtrees rather than vertex sets, and proving such a result result will be the starting point of our argument.
	
	\textbf{Formal Details}.  We begin with some technical definitions.
	
	\begin{definition}
		Given a tree $T$, we say that a subgraph $T'\sub T$ is a \textbf{subtree} of $T$ if $T'$ itself is a tree.  We say that a subtree $T'$ is \textbf{leaf-cuttable} if $e(T')\ge 1$ and if every edge in $E(T)\setminus E(T')$ which is incident to a vertex of $T'$ is incident to a leaf of $T'$.  We say that a collection of subtrees $\c{T}$ of a tree $T$ is \textbf{$k$-Helly} if for any $T_1,\ldots,T_k\in \c{T}$, there exists some $i\ne j$ such that $E(T_i)\cap E(T_j)\ne \emptyset$.  Here we adopt the convention that $\c{T}$ is $1$-Helly if and only if $\c{T}=\emptyset$.
	\end{definition}
	
	The name ``leaf-cuttable'' is meant to convey the fact that if $L'$ is the set of leaves of $T'$, then $T-L'$ contains $T'-L'$ as a connected component (that is, $T'$ can be ``separated'' from the rest of $T$ by cutting off its leaves).  With these definitions we can prove our Helly-type result for subtrees with respect to edges.
	\begin{theorem}\label{thm:Helly}
		Let $k\geq 1$ be an integer. If $T$ is a tree and if $\c{T}$ is a $k$-Helly collection of leaf-cuttable subtrees, then there exist a set $E$ of at most $k-1$ edges of $T$ such that every $T_i\in \c{T}$ contains at least one edge of $E$.
	\end{theorem}
	We note that this result is entirely false if we do not impose that the subtrees are leaf-cuttable.  Indeed, if $T=K_{1,2t-1}$ and if $\c{T}$ is the set of copies of $K_{1,t}$ in $T$, then one can check that $\c{T}$ has the 2-Helly property, but any set of edges $E$ which intersects every subtreee of $\c{T}$ must have size at least $t$ (which is in stark contrast to \Cref{thm:Helly} which ensures only 1 edge is required if the subtrees were leaf-cuttable).
	
	\begin{proof}
		Assume to the contrary that there exists a counterexample. Out of all such counterexamples let $k$, $T$ and $\mathcal{T}$ be chosen to minimize $k$, and then subject to this, minimize $v(T)$. We may assume that $k\ge 2$ and $v(T)\geq 3$ as the result is trivial if $k=1$ or if $v(T)\leq 2$. Let $x$ be an arbitrary leaf of $T$ and $y$ its unique neighbor. 
		
		First consider the case that the single edge tree $xy$ lies in $\c{T}$ and define a set of subtrees of $T$ by $\c{T}':=\{T'\in \c{T}:x\notin V(T')\}$.  We observe that $\c{T}'$ has the $(k-1)$-Helly property, since if there existed trees $T_1,\ldots,T_{k-1}\in \c{T}'$ which were all pairwise disjoint, then the $k$ trees $T_1,\ldots,T_{k-1},xy\in \c{T}$ would also be pairwise disjoint by definition of $\c{T}$, a contradiction to $\c{T}$ having the $k$-Helly property.  Because $\c{T}'$ has the $(k-1)$-Helly property, we have by our choice of minimal counterexample that there exist a set of $k-2$ edges of $T$ such that every tree of $\c{T}'$ contains at least one of these edges.  Adding $xy$ to this set of $k-2$ edges gives a set of $k-1$ edges such that every tree of $\c{T}$ contains at least one of these edges, proving the result in this case.
		
		Now consider the case that $xy\notin \c{T}$ and define $\c{T}':=\{T'-x:T'\in \c{T}\}$ which we view now as a collection of subtrees of $T-x$.  Note that by hypothesis of $xy\notin \c{T}$ that each of the trees in $\c{T}'$ has at least one edge. Furthermore, we claim for each $T'\in \c{T}$ that $T'-x$ is leaf-cuttable in $T-x$. Indeed, any edge $e$ in $E(T-x)\setminus E(T'-x)$ must also be in $E(T)\setminus E(T')$, and since $T'$ is leaf-cuttable, $e$ must be incident to a leaf of $T'$. This leaf cannot be $x$, so $e$ is also incident to a leaf of $T'-x$.
		
		We claim that $\c{T}'$ has the $k$-Helly property.  Indeed, assume for contradiction that there existed trees $T_1'-x,\ldots,T_k'-x\in \c{T}'$ which were pairwise edge disjoint with $T'_i\in \c{T}$ for all $i$.  Because $\c{T}$ had the $k$-Helly property, it must be that some $T'_i,T'_j$ share an edge, and because $T'_i-x$ and $T'_j-x$ are assumed to be edge disjoint, the edge they share must be $xy$. Furthermore, since $xy\not\in \mathcal{T}$, $y$ is not a leaf in either $T'_i$ or $T'_j$.  Let $z$ be a neighbor of $y$ in $T-x$. Since $y$ is not a leaf in $T'_i$, $z$ must be in $V(T'_i)$, as otherwise $yz$ would be in $E(T)\setminus E(T'_i)$ but not adjacent to a leaf of $T'_i$, contradicting the fact that $T'_i$ is leaf-cuttable. Similarly we must have $z\in V(T'_j)$. Thus, $yz\in E(T'_i-x)\cap E(T'_j-x)$, so $\mathcal{T}'$ has the $k$-Helly property.
		
		Since $T$ is a smallest counterexample with respect to $k$ and since $\mathcal{T}$ is a $k$-Helly collection of leaf-cuttable subtrees of $T-x$, there exists a set $E$ of at most $k-1$ edges of $T-x$ such that each $T'-x\in \mathcal{T}'$ contains an edge of $E$. This same set $E$ intersects every tree in $\mathcal{T}$, contradicting the assumption that $k$, $T$ and $\mathcal{T}$ give a counterexample.
	\end{proof}
	
	We now wish to define the set of subtrees $\c{T}$ for which we will apply \Cref{thm:Helly} to, and this will require a number of technical definitions motivated in part by the partition of $V(G)$ guaranteed by \Cref{lemma:canonicalReduction}.  
	
	\begin{definition}
		Let $T$ be a tree, $G$ a graph, and $\{V_{\hat{x}}:\hat{x}\in V(T)\}$ a partition of $V(G)$ indexed by $V(T)$.
		\begin{itemize}
			\item For each $u\in V(G)$, we define $\hat{u}$ to be the unique vertex in $T$ such that $u\in V_{\hat{u}}$, and for $S\sub V(G)$ we define $\widehat{S}=\{\hat{u}:u\in S\}$.  To avoid potential confusion with notation, we will by default denote generic vertices of $G$ by $u,v,w$ and generic vertices of $T$ by $\hat{x},\hat{y},\hat{z}$.  
			\item We say that a map $\psi:X\to V(G)$ with $X\sub V(T)$ is \textbf{canonical} if $\psi(\hat{x})\in V_{\hat{x}}$ for all $\hat{x}\in X$.  Note that canonical maps are necessarily injective since the $V_{\hat{x}}$ sets are pairwise disjoint. 
			\item If $S\sub V(G)$ is a set and if $T'\sub T$ is a subtree, then we define $\Psi(T';S)$ to be the set of canonical monomorphisms $\psi:V(T')\to V(G)$ which have $S$ in the image of $\psi$.
		\end{itemize}
	\end{definition}
	Here we think of $\Psi(T';S)$ as encoding the ways that one can ``extend'' the set $S$ into a copy of $T'$. With this we can formally define the set of subtrees that we want.
	
	\begin{definition}
		Let $T$ be a tree, $G$ a graph, $\{V_{\hat{x}}:\hat{x}\in V(T)\}$ a partition of $V(G)$ indexed by $V(T)$, and $k,\pow\ge 2$ a pair of integers.  We say that a subtree $T'\sub T$ is a \textbf{highly-extendable} with respect to some set $S\sub V(G)$ if the following conditions hold:
		\begin{enumerate}[label=(\alph*)]
			\item Every vertex of $\widehat{S}\sub V(T)$ is a leaf of $T'$,\label{condition highly extendable definition (a) leaf of T'}
			\item Every edge of $E(T)\sm E(T')$ which is incident to a vertex of $T'$ is incident to a vertex of $\widehat{S}$, \label{condition highly extendable defintion (b) edges incident}
			\item We have $|\Psi(T';S)|\geq \Pow(T',q)$ where $\Pow(T',q)$ is the real number guaranteed by \Cref{lemma:sunflower}.\label{condition highly extendable definition (c) many canonical morphisms}
		\end{enumerate}
		For each canonical monomorphism $\phi:V(T)\to V(G)$, we define $\c{T}_\phi$ to be the set of all subtrees $T'$ which are highly-extendable with respect to some non-empty subset of the image of $\phi$.
	\end{definition}
	
	The key technical result regarding these subtrees is the following.
	
	\begin{proposition}\label{proposition:HellyTechnical}
		Let $T\ne K_1$ be a tree, $G$ a graph, $\{V_{\hat{x}}:\hat{x}\in V(T)\}$ a partition of $V(G)$ indexed by $V(T)$, and $k,\pow \ge 2$ a pair of integers.  If $G$ is $\c{F}_{T,k}^\pow$-free and if $\phi:V(T)\to V(G)$ is a canonical monomorphism, then every $T'\in \c{T}_\phi$ is leaf-cuttable and $\c{T}_\phi$ has the $k$-Helly property.
	\end{proposition}
	We postpone the proof of this result and show how it can be used to prove our main result for this section.
	
	\begin{proof}[Proof of \Cref{theorem:treeBoundEG} assuming \Cref{proposition:HellyTechnical}]
		Let $G$ be an $n$-vertex $\c{F}_{T,k}^\pow$-free graph. By \Cref{lemma:canonicalReduction} there exists a partition of $V(G)$ into sets $\{V_{\hat{x}}:\hat{x}\in V(T)\}$ and a subset $\Phi\subseteq \mathrm{Mon}(T,G)$ with $|\Phi|=\Theta(\mathrm{mon}(T,G))$ such that $\phi(\hat{x})\in V_{\hat{x}}$ for all $\phi\in \Phi$ and $\hat{x}\in V(T)$, i.e.\ such that each $\phi\in \Phi$ is canonical.  Our ultimate goal will be to show that $\mathrm{mon}(T,G)=O(e(G)^{k-1})$, which by \Cref{observation monomorphisms same as copies} implies that $G$ contains at most this many copies of $T$ as desired. Since $|\Phi|=\Theta(\mathrm{mon}(T,G))$, it will suffice to prove $|\Phi|=O(e(G)^{k-1})$.  
		
		By Proposition~\ref{proposition:HellyTechnical} and \Cref{thm:Helly}, for every $\phi\in \Phi$, there exists a set of at most $k-1$ edges $\widehat{E}_\phi\sub E(T)$ such that every $T'\in \c{T}_\phi$ contains an edge of $\widehat{E}_\phi$, and without loss of generality we may assume that $\widehat{E}_\phi$ is non-empty (since if $\widehat{E}_\phi$ were empty we could add an arbitrary edge of $T\ne K_1$ to it while keeping its size $1\le k-1$).  Define 
		\[
		E_\phi:=\{\phi(\hat{x})\phi(\hat{y}):\hat{x}\hat{y}\in \widehat{E}_\phi\},
		\] 
		noting that $E_\phi$ is a non-empty set of at most $k-1$ edges of $E(G)$ since $\phi$ is a homomorphism, and that for every $T'\in \c{T}_\phi$ there exists an edge $uv\in E_\phi$ with $\hat{u}\hat{v}\in E(T')$. 
		
		We aim to show that for each non-empty set of at most $k-1$ edges $E\sub E(G)$, the number of $\phi\in \Phi$ with $E_{\phi}=E$ is $O(1)$, from which it will follow that $|\Phi|=O(e(G)^{k-1})$ since there are at most $e(G)^{k-1}$ choices for $E$.
		
		From now on we fix some $E$ as above.	Let $V_E\sub V(G)$ denote the set of vertices which are contained in at least one edge of $E$, let $T'_1,\ldots,T'_c$ denote the connected components of $T-\widehat{V}_E$, let $S_j\sub V_E$ denote the set of vertices $u\in V_E$ such that $\hat{u}$ is adjacent to a vertex of $T'_j$, and let $T''_j$ denote the subtree consisting of $T'_j$ and $\widehat{S}_j$.   We begin by observing some basic facts.
		
		\begin{claim}\label{cl:ETFacts}
			The following holds for each $j\ge 1$:
			\begin{enumerate}[label=(\roman*)]
				\item The set $S_j$ is non-empty,\label{claim ETFacts (i) non-empty S_j}
				\item Each $\hat{x}\in \widehat{S}_j$ is a leaf in $T''_j$\label{claim ETFacts (ii) each element of hat S_j is a leaf}
				\item No edge $uv\in E$ has $\hat{u}\hat{v}\in E(T''_j)$,\label{claim ETFacts (iii) no edge is in E(T''_j)}
				\item If $|\Psi(T''_j;S_j)|\ge \Pow$, then $T''_j$ is a highly-extendable subtree with respect to $S_j$.\label{claim ETFacts (iv) highly extentable}
			\end{enumerate}
		\end{claim}
		\begin{proof}
			For~\ref{claim ETFacts (i) non-empty S_j}, since $E$ is non-empty we have that $V_E$ is non-empty.  This means that each connected component of $T-\widehat{V}_E$ is a proper subgraph of $T$, and in particular the connected component $T'_j$ must have some vertex in $\widehat{V}_E$ which is adjacent to it in $T$, showing that $S_j$ is non-empty.

			For~\ref{claim ETFacts (ii) each element of hat S_j is a leaf}, we have by definition that each $\hat{x}\in \widehat{S}_j$ is adjacent to at least one vertex $\hat{y}$ of $T'_j\sub T''_j$.  If $\hat{x}$ were also adjacent to some $\hat{z}\ne \hat{y}$ in $T'_j$, then since $T'_j$ is connected there would exist a cycle in $T''_j\sub T$ formed by taking the path from $\hat{y}$ to $\hat{z}$ in $T'_j$ together with the edges from $\hat{x}$ to $\hat{y},\hat{z}$, a contradiction to $T$ being a tree.  Similarly if there existed some $\hat{x}'\in \widehat{S}_j$ which $\hat{x}$ was adjacent to in $T''_j$, then since $\hat{x}'\in \widehat{S}_j$ there exists some neighbor $\hat{y}'$ of $\hat{x}'$ in $T'_j$, and now the (possibly empty) path from $\hat{y}'$ to $\hat{y}$ in $T'_j$ together with the vertices $\hat{x},\hat{x}'$ would create a cycle in $T'_j$, giving a contradiction.

			For~\ref{claim ETFacts (iii) no edge is in E(T''_j)}, if $uv\in E$, then $\hat{u},\hat{v}\not\in V(T'_j)\sub V(T)\sm \widehat{V}_E$ by definition.  As such, having $\hat{u}\hat{v}\in E(T_j'')$ would imply that $\hat{u}\in \widehat{S}_j$ is not a leaf (since $\hat{u}$ is adjacent to some vertex in $T'_j$ by definition of $S_j$ in addition to the vertex $\hat{v}$), a contradiction to what we proved in~\ref{claim ETFacts (ii) each element of hat S_j is a leaf}.
			
			For~\ref{claim ETFacts (iv) highly extentable}, condition~\ref{condition highly extendable definition (c) many canonical morphisms} of being highly-extendable holds by hypothesis, and condition~\ref{condition highly extendable definition (a) leaf of T'} holds by~\ref{claim ETFacts (ii) each element of hat S_j is a leaf}. For~\ref{condition highly extendable defintion (b) edges incident}, let $\hat{u}\hat{v}\in E(T)\setminus E(T_j'')$ be an edge incident to $T_j''$, say with $\hat{u}\in V(T_j'')$, $\hat{v}\not\in V(T_j'')$. Assume to the contrary that $\hat{u}\not\in \widehat{S}$, so $\hat{u}\in V(T_j')$. If $\hat{v}\in \widehat{V}_E$, then $\hat{v}\in \widehat{S}\subseteq V(T_j'')$ by definition, a contradiction. If on the other hand $\hat{v}\not\in \widehat{V}_E$, then actually $\hat{u}$ and $\hat{v}$ must be in the same component of $T-\widehat{V}_E$, namely $T'_j$, a contradiction to $\hat{u}\hat{v}\notin E(T''_j)$.
		\end{proof}
		We next observe that the situation in \Cref{cl:ETFacts}\ref{claim ETFacts (iv) highly extentable} never occurs in the cases we care about.

		\begin{claim}\label{cl:ETSmall}
			If there exists some $\phi\in \Phi$ such that $E=E_\phi$, then $|\Psi(T''_j;S_j)|<\Pow(T''_j,q)$ for all $j$.
		\end{claim}
		
		\begin{proof}
			If this were not the case for some $j$, then $T''_j$ would be a highly-extendable subtree with respect to $S_j$ by \Cref{cl:ETFacts}\ref{claim ETFacts (iv) highly extentable}.  Observe that having $E=E_\phi$ in particular implies (by definition of $E_\phi$) that $V_E$ is in the image of $\phi$. Since $S_j\sub V_E$ is a non-empty subset of the image of $\phi$ by \Cref{cl:ETFacts}\ref{claim ETFacts (i) non-empty S_j}, we conclude that $T''_j\in\c{T}_\phi$ by definition.  However, by definition of $E_\phi$ there must exist some edge $uv\in E_\phi$ with $\hat{u}\hat{v}\in E(T''_j)$, but this does not hold by \Cref{cl:ETFacts}\ref{claim ETFacts (iii) no edge is in E(T''_j)}.
			
		\end{proof}
		
		Let $\displaystyle\Pow=\max_{T'\sub T} \Pow(T',q)$.  We claim that the number of $\phi\in \Phi$ with $E=E_\phi$ is at most $\Pow^{v(T)}=O(1)$.  This is certainly true if no such $\phi$ exists, so we may assume that at least one such $\phi$ exists.  For such $\phi$, let $\phi_j:=\phi|_{V(T''_j)}$ for $j\ge 1$ and let $\phi_0:=\phi|_{\widehat{V}_E}$.  Observe that $\phi$ is uniquely determined by the vector $(\phi_0,\phi_1,\ldots,\phi_c)$ since every $\hat{x}\in V(T)$ is in either $\widehat{V}_E$ or one of the sets $V(T'_j)\subseteq V(T''_j)$ by definition of $T'_j$ being the connected components of $T-\widehat{V}_E$.  Also observe that $\phi_0$ is uniquely determined by $E$ since $\phi$ is canonical, and that each $\phi_j$ with $j\ge 1$ is an element of $\Psi(T''_j;S_j)$.  Thus by \Cref{cl:ETSmall}, there are at most $\Pow^{c}\le \Pow^{v(T)}$ choices of vectors $(\phi_0,\ldots,\phi_c)$, and hence at most this many choices of $\phi$, proving the claim.  It follows that $|\Phi|=O(e(G)^{k-1})$, giving the result.
	\end{proof}
	
	We now prove our man technical statement assumed in the proof of \Cref{theorem:treeBoundEG}.
	
	\begin{proof}[Proof of \Cref{proposition:HellyTechnical}]
		We begin by verifying that each element in $\c{T}_\phi$ is leaf-cuttable.  Consider some $T'\in \c{T}_\phi$, which by definition means it is a highly-extendable subtree with respect to some non-empty subset $S$ of the image of $\phi$.  
		
		First, note that if we had $e(T')<1$, then the subtree $T'$ would be a single vertex $\hat{x}$.  Because each $\psi \in \Psi(T';S)$ is in particular a map $\psi:V(T')\to V(G)$ with the non-empty set $S$ in its image, we must have $|\Psi(T';S)|\le 1$, a contradiction to condition~\ref{condition highly extendable definition (c) many canonical morphisms} of $T'$ being highly-extendable since by \Cref{lemma:sunflower}, we have $\Pow(T',q) \ge \pow\ge 2$.  We conclude that every $T'\in \c{T}_\phi$ has at least one edge.
		
		Second, by conditions~\ref{condition highly extendable definition (a) leaf of T'} and~\ref{condition highly extendable defintion (b) edges incident} for $T'$ being highly-extendable with respect to $S$, we have that every edge of $E(T)\sm E(T')$ which is incident to $T'$ is incident to a leaf of $T'$ (namely to a vertex of $\widehat{S}$).  This combined with the observation above implies that $T'$ is always a leaf-cuttable subtree. 
		
		It remains now to verify that $\c{T}_\phi$ is $k$-Helly.  To this end, we assume for contradiction that there exist some canonical homomorphism $\phi$ and $T_1',\ldots,T'_k\in \c{T}_\phi$ which are pairwise edge disjoint.  
		
		For each $j\ge 1$, let $S_j$ be a minimal (and possibly empty) set of vertices in the image of $\phi$ such that $T'_j$ is highly-extendable with respect to $S_j$, and this minimum is well-defined since having $T'_j\in \c{T}_\phi$ means $T'_j$ is highly-extendable with respect to at least one set.  For each $j\ge 1$, let $X_j\sub V(T)$ be the set of vertices $\hat{x}\in V(T)$ with the property that every edge containing $\hat{x}$ in $T$ lies in $T_j$, and also let $X_0=V(T)\sm \bigcup_{j\ge 1} X_j$.
		\begin{claim}\label{cl:XFacts}
			The following hold:
			\begin{enumerate}[label=(\roman*')]
				\item The $X_j$ sets for $j\ge 0$ partition $V(T)$.\label{cl:XFacts (i') partition}
				\item $\widehat{S}_j\sub X_0$ for all $j\ge 1$.\label{cl:XFacts (ii') hat S_j is in X_0}
				\item $X_j\cup \widehat{S}_j=V(T'_j)$ for all $j\ge 1$.\label{cl:XFacts (iii') X_j cup hat S_j is T'_j}
			\end{enumerate}
		\end{claim}
		
		\begin{proof}
			For~\ref{cl:XFacts (i') partition}, the fact that the union of these sets equals $V(T)$ follows from the definition of $X_0$, which also implies $X_0$ is disjoint from every $X_j$.  If there existed some $i\ne j$ with $i,j\ge 1$ and a vertex $\hat{x}\in X_i\cap X_j$, then consider any edge of the form $\hat{x}\hat{y}\in E(T)$ (and such an edge always exists since $T\ne K_1$).  By definition of $\hat{x}\in X_i\cap X_j$ with $i,j\ge 1$, we must have that $\hat{x}\hat{y}$ is in both $T'_i$ and $T'_j$, a contradiction to our assumption that these trees are edge disjoint.  We conclude that these sets partition $V(T)$.
			
			For~\ref{cl:XFacts (ii') hat S_j is in X_0}, each $u\in S_j$ must have $\hat{u}$ contained in an edge of $T'_j$ by condition~\ref{condition highly extendable definition (a) leaf of T'} of $T'_j$ being highly-extendable with respect to $S_j$, so $\hat{u}\notin X_i$ for any $i\ne 0,j$.  If $\hat{u}\in X_j$, then we claim that $T'_j$ would be highly-extendable with respect to $S_j\sm \{u\}$.  Indeed, condition~\ref{condition highly extendable definition (a) leaf of T'} trivially holds since it held for $S_j$,~\ref{condition highly extendable defintion (b) edges incident} continues to hold since $\hat{u}\in X_j$ is assumed to only be in edges of $T'_j$, and~\ref{condition highly extendable definition (c) many canonical morphisms} holds since $\Psi(T'_j;S_j\sm \{u\})\supseteq \Psi(T'_j;S)$.  However, $T'_j$ can not be highly-extendable with respect to $S_j\sm \{u\}$ since we chose $S_j$ to be a minimal set with the property, a contradiction.  We conclude that $\hat{u}\notin X_j$, so we must have $\hat{u}\in X_0$ by~\ref{cl:XFacts (i') partition}.
			
			For~\ref{cl:XFacts (iii') X_j cup hat S_j is T'_j}, the inclusion $X_j\cup \widehat{S}_j\sub V(T'_j)$ follows from the fact that each vertex in $X_j$ is incident to at least one edge of $T'_j$ by definition, with this also holding for $\widehat{S}_j$ by condition~\ref{condition highly extendable definition (a) leaf of T'} of $T'_j$ being highly-extendable with respect to $S_j$.  On the other hand, every vertex $\hat{x}\in V(T'_j)\sm \widehat{S}_j$ can only be incident to edges of $T'_j$ by condition~\ref{condition highly extendable defintion (b) edges incident} of $T'_j$ being highly-extendable with respect to $S_j$, implying $\hat{x}\in X_j$ as desired.
		\end{proof}
		
		Define the map $\psi_0:X_0\to V(G)$ by having $\psi_0(\hat{x})=\phi(\hat{x})$ for all $\hat{x}\in X_0$, noting that this map is a canonical monomorphism of $T[X_0]$ into $G$ since $\phi$ is a canonical monomorphism of $T$.  We say that a tuple of maps $M=(\psi_0,\psi_1,\ldots,\psi_k)$ is \textbf{gluable} if $\psi_j\in \Psi(T'_j;S_j)$ for all $j\ge 1$.  We define the \textbf{amalgamation} $\psi$ of a gluable tuple of maps $M$ to be the map $\psi:V(T)\to V(G)$ with $\psi(\hat{x})=\psi_j(\hat{x})$ whenever $\hat{x}\in X_j$.  Note that this map is well defined since the $X_j$ sets partition $V(T)$ by \Cref{cl:XFacts}\ref{cl:XFacts (i') partition}. Furthermore, for $\hat{x}\in \widehat{S}_j\subseteq X_0$, we have that $\psi_j(\hat{x})=\phi(\hat{x})=\psi(\hat{x})$ since $\psi_j$ and $\phi$ are canonical and both have $S_j$ in their image.
		
		\begin{claim}\label{cl:amalag}
			The amalgamation $\psi$ of any gluable tuple of maps $M$ is a canonical monomorphism from $T$ to $G$.
		\end{claim}
		\begin{proof}
			The fact that $\psi$ is canonical (and thus injective) follows from the fact that each $\psi_j$ is canonical by construction.  Similarly, if $\hat{x}\hat{y}\in E(T)$ is such that $\hat{x},\hat{y}\in X_j$ for some $j$, then $\psi$ will map $\hat{x}\hat{y}$ to an edge of $G$ since $\psi_j$ is a homomorphism.  Thus to verify that $\psi$ map edges of $T$ to edges of $G$, we need only consider the case that $\hat{x}\hat{y}\in E(T)$ with $\hat{x}\in X_i$ and $\hat{y}\in X_j$ for some $i\ne j$, and without loss of generality we may assume $j\ne 0$.  Because $\hat{y}\in X_j$ with $j\ne 0$, every edge incident to $\hat{y}$ in $T$ must lie in $T'_j$, so in particular $\hat{x}\in V(T'_j)$.  Since $\hat{x}\notin X_j$ by assumption, we have by \Cref{cl:XFacts}\ref{cl:XFacts (iii') X_j cup hat S_j is T'_j} and~\ref{cl:XFacts (ii') hat S_j is in X_0} that $\hat{x}\in V(T'_j)\sm X_j=\widehat{S}_j$. Recall that for $\hat{x}\in \widehat{S}_j$, we have $\psi(\hat{x})=\psi_j(\hat{x})$. Because $\psi_j$ is a homomorphism, we conclude that $\psi(\hat{x})\psi(\hat{y})=\psi_j(\hat{x})\psi_j(\hat{y})$ is an edge of $G$, proving that $\psi$ is indeed a homomorphism.
		\end{proof}
		
		Observe that each $\psi_j\in \Psi(T'_j;S_j)$ is a monomorphism of $T'_j$ and that $|\Psi(T'_j;S_j)|\ge \Pow(T'_j,q)$ since $T'_j$ is highly-extendable with respect to $S_j$.  Thus by \Cref{lemma:sunflower}, for each $j\ge 1$ there exists a set $R_j\subsetneq V(T'_j)$ and distinct maps $\psi_j^{(1)},\ldots,\psi_j^{(\pow)}\in \Psi(T'_j;S_j)$ such that each pair of maps $\psi_j^{(i)},\psi_j^{(i')}$ with $i\ne i'$ has $\psi_j^{(i)}(\hat{x})=\psi_j^{(i')}(\hat{x})$ if and only if $\hat{x}\in R_j$.  Note that we necessarily have $\widehat{S}_j\sub R_j$ for all $j$ since every map in $\Psi(T'_j;S_j)$ agrees on $\widehat{S}_j$.  For each $1\le i\le \pow$, let $\psi^{(i)}$ be the amalgamation of $M^{(i)}:=\{\psi_0,\psi_1^{(i)},\ldots,\psi_k^{(i)}\}$. 
		\begin{claim}\label{claim amalgamation}
			The following holds:
			\begin{enumerate}[label=(\roman*'')]
				\item Each $\psi^{(i)}$ is a monomorphism from $T$ to $G$ for all $1\le i\le \pow$.\label{claim amalgamation (i'') monomorphism}
				\item For all $i\ne i'$ and $\hat{x},\hat{y}\in V(T)$, we have $\psi^{(i)}(\hat{x})=\psi^{(i')}(\hat{y})$ if and only if both $\hat{x}=\hat{y}$ and $\hat{x}\in R:=X_0\cup \bigcup_{j\ge 1} R_j$.\label{claim amalgamation (ii'') well-defined}
				\item $T-R$ has at least $k$ connected components.\label{claim amalgamation (iii'') T-R has k components}
			\end{enumerate}
		\end{claim}
		\begin{proof}
			\ref{claim amalgamation (i'') monomorphism} follows from \Cref{cl:amalag}.
			
			For the reverse direction of~\ref{claim amalgamation (ii'') well-defined}, if $\hat{x}=\hat{y}$ lies in some $R_j\sub X_j\cup \widehat{S}_j$, then by definition of the amalgamation and $R_j$,  \[\psi^{(i)}(\hat{x})=\psi_j^{(i)}(\hat{x})=\psi_j^{(i')}(\hat{x})=\psi^{(i')}(\hat{x}),\] with a similar result holding if $\hat{x}=\hat{y}\in X_0$.  
			
			Now, for the forward direction of~\ref{claim amalgamation (ii'') well-defined}, assume $\hat{x},\hat{y}$ satisfy $\psi^{(i)}(\hat{x})=\psi^{(i')}(\hat{y})$.  Because amalgamations are canonical, we have $\psi^{(i)}(\hat{x})\in V_{\hat{x}}$ and $\psi^{(i)}(\hat{y})\in V_{\hat{y}}$, and since $V_{\hat{x}},V_{\hat{y}}$ are disjoint for $\hat{x}\ne \hat{y}$, we must have $\hat{x}=\hat{y}$. If $\hat{x}\notin X_0$, then by \Cref{cl:XFacts}\ref{cl:XFacts (i') partition}, we have $\hat{x}\in X_j$ for some $j\ge 1$. Then $\psi^{(i)}(\hat{x})=\psi^{(i')}(\hat{x})$ implies $\psi^{(i)}_j(\hat{x})=\psi^{(i')}_j(\hat{x})$, meaning $\hat{x}\in R_j$ by definition of $R_j$.  We conclude that $\hat{x}\in X_0\cup \bigcup_{j\ge 1} R_j=R$, proving the result.
			
			For~\ref{claim amalgamation (iii'') T-R has k components}, let $\hat{x}_j$ be an arbitrary vertex in $V(T'_j)\sm R_j$ for each $j\ge 1$, which exists since $R_j$ is a proper subset of $V(T'_j)$.  We aim to show that each $\hat{x}_j$ lies in a different connected component of $T-R$, from which the result will follow.  
			
			Indeed, consider the path $\hat{z}_1\cdots \hat{z}_a$ in $T$ from some $\hat{x}_j$ to some other $\hat{x}_{j'}$ with $j\ne j'$.  Note that $\hat{z}_1=\hat{x}_j$ lies in $V(T'_j)\sm \widehat{S}_j$ since $\widehat{S}_j\sub R_j$, which means $\hat{z}_1\in X_j$ by \Cref{cl:XFacts}\ref{cl:XFacts (iii') X_j cup hat S_j is T'_j}.  Similarly $\hat{z}_a\in X_{j'}$, and since $X_j,X_{j'}$ are disjoint, there must exist some smallest integer $b>1$ such that $\hat{z}_b\notin X_j$. Because $\hat{z}_{b-1}\in X_j$ by the minimality of $b$, and because $\hat{z}_{b-1}\hat{z}_b$ is an edge of $T$, we have by definition of $\hat{z}_{b-1}\in X_j$ that $\hat{z}_{b-1}\hat{z}_b$ must be an edge in $T'_j$.  Since $\hat{z}_b\notin X_j$, we must have \[\hat{z}_b\in V(T'_j)\sm X_j= \widehat{S}_j\sub R_j\sub R.\]  We conclude  that the paths in $T$ between the $\hat{x}_j$ vertices always contains at least one vertex in $R$, implying that each of these vertices lies in a different connected component of $T-R$, proving the result.
		\end{proof}
		In total, this last claim implies that $G$ contains a copy of some $T_R^\pow\in \c{F}_{T,k}^\pow$, contradicting $G$ being $\c{F}_{T,k}^\pow$-free.  We conclude that no set of edge disjoint $T_1',\ldots,T_k'$ exist in $\c{T}_\phi$, proving the result.
	\end{proof}

	\section{Concluding Remarks}\label{section concluding remarks}
	As things stand, it is not clear what the set of realizable exponents for an arbitrary graph should be.  At present, our only guess for what might be the right answer in general is the following, where here we recall that $\c{F}_{H,k}^\pow$ is the family introduced in \Cref{definition flower powers}.
	\begin{question}\label{question:obviousObstruction}
		Is it the case that for all graphs $H$ and integers $k\in [1,v(H)]$, a rational number $\re\in (k-1,k)$ is realizable for $H$ if and only if $\ex(n,H,\c{F}_{H,k}^\pow)=\Om(n^\re)$ for some sufficiently large $\pow$?
	\end{question}
	Recall that in \Cref{corollary:keyObservation} we showed that if $\ex(n,H,\c{F}_{H,k}^\pow)=o_\pow(n^\re)$ then $\re$ is not realizable, so \Cref{question:obviousObstruction} in essence asks if the converse of this statement is true.  We believe we can prove that \Cref{question:obviousObstruction} has a positive answer for all $H$ on at most 4 vertices, but other than this we have little reason to believe the answer to \Cref{question:obviousObstruction} is yes in general.  Still, this question raises a number of followup questions that may be easier to handle, such as the following.
	
	\begin{question}\label{question:independenceNumber}
		Is it the case that for every graph $H$, every rational in $[\al(H),v(H)]$ is realizable, where here $\al(H)$ denotes the independence number of $H$?
	\end{question}
	A positive answer to \Cref{question:obviousObstruction} would imply a positive answer to \Cref{question:independenceNumber}, as we showed in \Cref{lemma alpha empty} that $\c{F}_{H,k}^\pow=\emptyset$ whenever $k> \al(H)$.  We also note that the bounds of this question exactly matches the cutoff for realizable exponents of both stars and cliques.
	
	While we are unwilling to go so far as to call \Cref{question:independenceNumber} a conjecture, we do believe a slight weakening of this should be true.
	
	\begin{conjecture}\label{conjecture:maxDegree}
		If $H$ is a graph of maximum degree $\Del(H)$, then every rational in the interval
		\[
		\left[v(H)-\frac{e(H)}{\Del(H)},v(H)\right]
		\]
		is realizable for $H$.
	\end{conjecture}
	
	A positive answer to \Cref{question:independenceNumber} would suffice to prove \Cref{conjecture:maxDegree}, as it is easy to show\footnote{This is because the complement of a (maximum) independent set $S$ must be incident to all $e(H)$ edges of $H$, but also any set of size $v(H)-\al(H)$ is trivially incident to at most $(v(H)-\al(H))\Del(H)$ edges of $H$.} that every graph satisfies $\al(H)\le v(H)-e(H)/\Del(H)$.  This conjecture holds for trees due to \Cref{theorem:maxDegreeTrees}, with \Cref{theorem:maxDegreeGeneral} giving an approximate version of this conjecture for all graphs.
	
	As far as we know, our general result \Cref{theorem d large tree implies dadmissible} may be sufficient to prove \Cref{conjecture:maxDegree} for all $H$, in which case it would prove the stronger fact that every graph is $d$-admissible for all $d\ge \Del$.  However, we emphasize that while $H$ being $d$-admissible implies $v(H)-e(H)/d$ is a realizable exponent, the converse of this statement is (perhaps surprisingly) not true in general.  That is, there are cases where a direct application of the random polynomial method from \Cref{theorem random polynomial lower bounds} can not be used to prove the realizability of exponents. 
	
	As an informal example of this phenomenon, consider the problem of showing that rationals in the range $(1.5,2)$ are realizable for the complete tripartite graph $H=K_{1,1,2}$.  Here the  issue  is that in order to have $\ex(n,H,\c{F})=o(n^2)$, the family $\c{F}$ must contain (a subgraph of) some $K_{1,1,t}$, as otherwise $G=K_{1,1,n-2}$ would be $\c{F}$-free and give $\Om(n^2)$ copies of $H$.  However, $K_{1,1,t}$ can be viewed as the rooted power of $(K_3,\{x,y\})$ with $xy$ an edge of the triangle, which has rooted density 2.  Thus when applying \Cref{theorem random polynomial lower bounds} to avoid a family containing large powers of $(K_3,\{x,y\})$, the number of copies of $K_{1,1,2}$ that one finds is at most on the order of $n^{4-5/2}=n^{1.5}$.  As such, a direct application of random polynomials can not be used to prove the existence of realizable exponents in $(1.5,2)$.  Nevertheless, it is in fact possible to show that every rational in $(1.5,2)$ is realizable for $H$.  We plan to do this in forthcoming work by starting with a random polynomial graph and then ``trimming'' it in some appropriate way.
	
	For our final question, we recall that \Cref{corollary trees vs forests} showed that $\ex(n,T,\c{F})=\Theta(n^k)$ for some integer $k$ whenever $T$ is a tree and $\c{F}$ is a forest.  We wonder if this same result might continue to hold for other choices of $H$.
	
	\begin{question}
		Is it true that for any graph $H$ and any finite family $\c{F}$ containing a forest that either $\ex(n,H,\c{F})=\Theta(n^k)$ for some integer $k$ or $\ex(n,H,\c{F})=0$ for all sufficiently large $n$?
	\end{question}
	Perhaps the most natural approach towards proving this question for a given $H$ would be to extend either \Cref{theorem weak tree reduction} or \Cref{theorem:treeReduction} to $H$. For example, if $k=\alpha(H)+1$ then \Cref{lemma alpha empty} gives that $\c{F}_{H,k}^q=\emptyset$, so the natural analog of \Cref{theorem:treeReduction} would suggest that for any family $\c{F}$ we have
	\[
	\ex(n,H,\c{F})=O(\ex(n,\c{F})^{\alpha(H)}).
	\]
	This in fact holds if $H$ is a bipartite graph without isolated vertices, since in this case K\"onig's Theorem implies that the vertices of $H$ can be covered by $\alpha(H)$ edges.  On the other hand, this inequality is false for graphs $H$ with $\alpha(H)<v(H)/2$ simply by considering $\c{F}=\emptyset$.  Because of this observation, one might hope that \Cref{theorem:treeReduction} could continue to hold with $T$ replaced by any bipartite graph $H$ without isolated vertices, but this is false for $H=C_4$ and $k=2$.  Indeed, in this case the natural generalization of \Cref{theorem:treeReduction} would essentially state that every family $\c{F}$ which contains $K_{2,q}$ satisfies $\ex(n,C_4,\c{F})=O(\ex(n,\c{F}))$.  However, \Cref{theorem random polynomial lower bounds} shows that for $\c{F}=\{K_{2,q}\}$ we have $\ex(n,C_4,\c{F})=\Omega(n^2)$ which is much larger than $\ex(n,\c{F})=O(n^{3/2})$.   
	
	From here one might hope that even if \Cref{theorem:treeReduction} is false for $C_4$ that perhaps \Cref{theorem weak tree reduction} still holds.  This again turns out to be false for $k=2$ but requires a bit more work to show, the details of which we briefly sketch below.
	
	\begin{proposition}
		There exists a family $\c{F}$ such that neither $\ex(n,C_4,\c{F})=\Om(n^2)$ nor the bound $\ex(n,C_4,\c{F})=O(\ex(n,\c{F}))$ holds.  In particular, the statement of \Cref{theorem weak tree reduction} does not hold for $C_4$ and $k=2$.
	\end{proposition}
	\begin{proof}[Sketch of Proof]
		Take $\c{F}$ to consist of $K_{2,q}$ together with the powers of a rooted tree $(T,R)$ of rooted density $d$ with $3/2<d<2$.  Note that we have $\ex(n,\c{F})=O(\ex(n,(T,R)^q))=O(n^{2-1/d})$ while \Cref{theorem random polynomial lower bounds} implies $\ex(n,C_4,\c{F})=\Omega(n^{4-4/d})$, and this is much larger than $\ex(n,\c{F})$ for $d>3/2$.  We claim now that $\ex(n,C_4,\c{F})=O(n^{2-\varepsilon})$ for some sufficiently small $\varepsilon$.  To this end, say we have an $n$-vertex $\c{F}$-free graph $G$ and let $S$ be its vertices of degree at most $n^{1-1/d+\varepsilon/2}$ and $L$ the rest of the vertices.  Because $G$ is $K_{2,q}$-free the number of copies of $C_4$ using a vertex of $S$ is at most $n^{1+2(1-1/d+\varepsilon/2)}\le n^{2-\varepsilon}$ provided $1-2/d+\varepsilon\le -\varepsilon$, which holds for $\varepsilon$ sufficiently small provided $d<2$.  On the other hand, we must have $|L|=O(n^{1-\varepsilon/2})$ since we know $e(G)\le \ex(n,\c{F})=O(n^{2-1/d})$, so again being $K_{2,q}$-free implies that the number of copies of $C_4$ using only vertices in $L$ is at most $q|L|^2=O(n^{2-\varepsilon})$ as desired.
	\end{proof}

	\subsection{Acknowledgments}
	We thank Robert A. Krueger for helpful conversations and D\'aniel Gerbner for pointing us to useful references.  The second author is supported by a National Science Foundation Mathematical Sciences Postdoctoral Research Fellowship under Grant No. DMS-2202730.
	
	\bibliographystyle{amsplain}
	\bibliography{bib}{}
	
\end{document}